\documentclass{amsart}
\usepackage{mathtools} 
\usepackage{amssymb}
\usepackage{booktabs}
\usepackage{mathrsfs}
\usepackage{psfrag}
\usepackage{xr-hyper}
\usepackage[colorlinks]{hyperref}
\subjclass[2010]{ 37C30,  30H25,  42B35, 42C15, 42C40} 

\keywords{ atomic decomposition, Besov space, harmonic analysis, wavelets}

\title{Classic and exotic Besov spaces induced by good grids}

\author[D.  Smania]{Daniel Smania}
\address{Departamento de Matem\'atica \\
  Intituto de Ci\^encias Matem\'aticas e da Computa\c{c}\~ao-Universidade de S\~ao Paulo (ICMC/USP) - S\~ao Carlos \\
               Caixa Postal 668 \\ S\~ao Carlos-SP \\ CEP 13560-970 \\ Brazil.}
\email{smania@icmc.usp.br} 
\urladdr{http://conteudo.icmc.usp.br/pessoas/smania/}
\thanks{D.S. was partially supported by CNPq 307617/2016-5, CNPq 430351/2018-6, CNPq 306622/2019-0 and FAPESP Projeto Tem\'atico 2017/06463-3.}

\newtheorem{theorem}{Theorem}[section]
\newtheorem{corollary}[theorem]{Corollary}

\newtheorem{lemma}[theorem]{Lemma}
\newtheorem{proposition}[theorem]{Proposition}

\theoremstyle{definition}
\newtheorem{definition}[theorem]{Definition}
\newtheorem{remark}[theorem]{Remark}

\usepackage{constants} 
\newcommand{\secdot}[1]{\arabic{#1}}
\newconstantfamily{c}{
symbol=\lambda,
format=\secdot,
}
\renewconstantfamily{normal}{
symbol=C,
format=\secdot,
}

\newconstantfamily{e}{
symbol=\theta,
format=\arabic,
}

\newconstantfamily{A}{
symbol=A,
format=\arabic,
}

\newconstantfamily{G}{
symbol=G,
format=\arabic,
}
\usepackage{ltxcmds}[2011/04/14]
\makeatletter

\newcommand{\Cll}[2][normal]{\Cl[#1]{#2}}
\newcommand{\Crr}[1]{\Cr{#1}}
\def\@CatchSuperScriptNegKern^#1{%
  ^{\!\!#1}%
}
\makeatother

\usepackage[colorinlistoftodos, textwidth=4cm, shadow]{todonotes}

\hypersetup{ colorlinks=true, filecolor={black}, pdftitle={Besovish spaces through  atomic decomposition}, pdfsubject={  atomic decomposition, Besov space}, pdfauthor={Daniel Smania},pdfkeywords={  atomic decomposition, Besov space, harmonic analysis, wavelets}}

\begin{document}

\begin{abstract} 
In a previous work we introduced Besov spaces $\mathcal{B}^s_{p,q}$ defined on a measure spaces with a good grid, with $p\in [1,\infty)$, $q\in [1,\infty]$ and $0< s< 1/p$. Here we show that classical Besov spaces on compact homogeneous spaces are examples of such Besov spaces. On the other hand we show that even Besov spaces defined  by a good grid made of partitions by  intervals   may differ from a classical Besov space, giving birth to exotic Besov spaces.\end{abstract}

\maketitle

\setcounter{tocdepth}{3}
\tableofcontents

\newpage
\section{Introduction}

We defined Besov spaces $\mathcal{B}^s_{p,q}$, with $p\in [1,\infty)$, $0< s< 1/p$ and $q\in [1,\infty]$, on   a  measure spaces with  a good grid \cite{smania-besov}.  Those spaces are defined by atomic decomposition using very simple atoms consisting  of  piecewise constant functions.

Similar but more general than  Gu-Taibleson \cite{martingale} recalibrated martingale Besov spaces, it allows us to carry   results for classic  Besov spaces in  $\mathbb{R}^n$ to  measure spaces with  a good grid, and often the proofs there are indeed simpler and more elementary. This includes results on multipliers, atomic decomposition  and left and right compositions.\\

\subsection{Measure spaces with good grids} 
A {\it measure space with a good grid}  is a set $I$ endowed with a $\sigma$-algebra $\mathbb{A}$ and   a measure $m$  on $(I,\mathbb{A})$, $m(I)<  \infty$.  For every measurable set $S$ denote $|S|=m(S)$. A {\bf good grid}  on $I$  is a sequence of finite  families of measurable sets with positive measure  $\mathcal{P}= (\mathcal{P}^k)_{k\in \mathbb{N}}$ so that
\begin{itemize}
\item[i.] Every family $\mathcal{P}^k$ is a partition of  $I$ up to sets of zero measure. 
\item[ii.] The family
$$\cup_k \mathcal{P}^k$$
generates the $\sigma$-algebra $\mathbb{A}$.
\item[iii.] There is $\lambda,\hat{\lambda}\in (0,1)$ such that 
$$\hat{\lambda} \leq \frac{|Q|}{|P|} \leq \lambda$$
for  all $Q\subset P$ such that  $Q\in \mathcal{P}^{k+1}$ and $P\in \mathcal{P}^{k}$ for some $k\geq 0$. 
\end{itemize}

\subsection{Besov space on measure spaces with good grid} 
  For each $Q\in \mathcal{P}$ consider the function  $a_Q$  defined by $$a_Q(x)=|Q|^{s-1/p}$$
for every $x \in Q$ and  $a_Q(x)=0$ otherwise. The function $a_Q$ is the  Souza's canonical atom on $Q$.  The {\it Besov space}  $\mathcal{B}^s_{p,q}$  is  the space of all  functions $f$ in the Lebesgue space  $L^p$ that  can  be represented by an absolutely convergent series on $L^p$
\begin{equation} \label{rep55} f = \sum_{k=0}^{\infty} \sum_{Q \in \mathcal{P}^k}    s_Q a_Q\end{equation}
where $s_Q \in \mathbb{C}$ and additionally 
\begin{equation} \label{rep2} \big( \sum_{k=0}^{\infty} (\sum_{Q \in \mathcal{P}^k}    |s_Q|^p)^{q/p} \big)^{1/q} < \infty.\end{equation}
The r.h.s. of (\ref{rep55}) is  a $\mathcal{B}^{s}_{p,q}$-representation of $f$. Define
\begin{equation} \label{norm} |f|_{\mathcal{B}^s_{p,q}} = \inf  \big( \sum_{k=0}^{\infty} (\sum_{Q \in \mathcal{P}_k}     |s_Q|^p )^{q/p} \big)^{1/q} ,\end{equation} 
where the infimum runs over all possible representations of $f$ as in (\ref{rep55}).    Then  $(\mathcal{B}^s_{p,q},|\cdot|_{\mathcal{B}^s_{p,q}})$ is  a complex Banach space  and its unit ball is compact in $L^p$ (see \cite{smania-besov}).

\subsection{Comparing Besov spaces} There is a large body of literature on  Besov  spaces.  See Stein \cite{stein}, Peetre \cite{peetre},  Triebel \cite{tbook} and the references therein.  Since Besov \cite{besov} defined Besov spaces $\mathbb{R}^n$ in late 50's, there is a long and ongoing quest to extend  Besov spaces (and indeed harmonic analysis)  to settings with weaker structure.  Han and  Sawyer \cite{hs} and   Han,  Lu and Yang  \cite{han2}  defined  Besov spaces on {\it homogeneous spaces}. Those are a large class of quasi-metric  spaces endowed with a doubling measure, introduced by Coifman and Weiss \cite{cw}.  There are also definitions of Besov spaces on metric spaces and $d$-sets. See  Alvarado and   Mitrea \cite{sharp} and Koskela, Yang, and Zhou \cite{koskela2} and  Triebel \cite{fractal}\cite{fractal2}. \\

Our goal here is to  {\it compare}  the Besov spaces of a homogeneous space with the ones defined on a measure spaces with a good grid. While the former is certainly a far-reaching  generalisation of Besov spaces on $\mathbb{R}^n$,   the latter  may at first glance looks like   an  artificial  {\it dyadic version} of the classical Besov spaces, a sort of simplistic model  of  more complex situations.  {\it  As it turn out,  this couldn’t be farther from the truth.} \\

Indeed there are earlier works that give  atomic decompositions to  classical Besov spaces where the atoms are piecewise constant functions. We cite the  atomic decomposition of the Besov space $B^s_{1,1}([a,b])$, with $s\in (0,1)$,  by de Souza \cite{souzao1} (see De Souza \cite{souzao2} and de Souza, O'Neil and Sampson  \cite{sn}), some  results on B-spline atomic decomposition of the Besov space  of the unit cube of $\mathbb{R}^n$ by  DeVore and  Popov \cite{spline2}, as well results   on finite element approximation in bounded polyhedral domains on $\mathbb{R}^n$ by Oswald \cite{oswald0} \cite{oswald}.\\

In part I, we consider compact homogeneous spaces, and using the famous dyadic "cubes" constructed by Christ \cite{christ} (see also Hyt\"onen and  Kairema \cite{hk}  for recent results on dyadic cubes in homogeneous spaces), we show that its Besov spaces, as defined by Han,  Lu and Yang  \cite{han2}, are a {\it particular case of the Besov spaces induced by grids}.   Note that   Yang \cite{yang} already proved that Besov spaces on $d$-sets defined by Triebel \cite{fractal} also coincide with Besov spaces on these sets considering them as homogeneous spaces. We also observe  that Gu-Taibleson recalibrated martingale Besov spaces \cite{martingale}  {\it are}  Besov spaces of certain compact  homogeneous spaces (see Section \ref{gu}). \\

\noindent {\it  Besov spaces on compact homogeneous spaces, with $p \in [1,\infty),q\geq 1$ and small $s$ satisfying $0 < s < 1/p$, and in particular recalibrated martingale Besov spaces,   can be retrieved as  particular examples of Besov spaces in measure spaces with a good grid. }
\vspace{5mm}

On the other hand in Part II we show that, even if we consider $[0,1]$ with a grid formed by  a nested sequence of partitions by  intervals, but that it  is not "recalibrated" as in Gu and Taibleson \cite{martingale}, then it may happens that the induced Besov space does {\it not } coincide with the Besov space of  $[0,1]$.  Those "exotic" Besov spaces are indeed useful to study certain transfer operators  associated to expanding maps \cite{smania-transfer}.\\

Yet  our construction of Besov spaces $B^s_{p,q}$ with low regularity $s$ is likely to be the  simplest construction  available.

\vspace{1cm}
\addcontentsline{toc}{chapter}{\bf I. THE BEST OF ALL POSSIBLE WORLDS.}

\centerline{ \bf I. THE BEST OF ALL POSSIBLE WORLDS. }
\vspace{1cm}

\section{Besov spaces on compact homogeneous spaces} \label{beshom}

\subsection{Homogeneous spaces}
Let $I$ be a set. A  quasi-metric $\rho(\cdot,\cdot)$ in  $I$ is a real-valued function in $I\times I$  satisfying 
\begin{itemize}
\item[HS1.] We have  $\rho(x,x)=0$,
\item[HS2.] If $x\neq y$ then $\rho(x,y) > 0$,
\item[HS3.] We have $\rho(x,y)=\rho(y,x)$,
\item[HS4.] There exists $\Cll{pm}\geq 1$ such that 
$$\rho(x,z)\leq \Crr{pm}( \rho(x,y)+\rho(y,z))$$
\end{itemize}
For every $\gamma > 0$ the function  $\rho^\gamma$ is also a quasi-metric on $I$. A  quasi-metric $\rho'$ is  equivalent to $\rho$ if  
$$\frac{1}{C} \rho'(x,y)  \leq  \rho(x,y)  \leq C \rho'(x,y) $$
for some $C\geq 1$. We say that two quasi-metrics $\rho$ and $\rho'$ are power-law equivalent if there exists $\gamma >0$ such that $\rho'$ is equivalent to $\rho^\gamma$.


A homogeneous space $(I,\rho,m)$, introduced by Coifman and Weiss \cite{cw},   is a topological space $I$ endowed with a quasi-metric $\rho(\cdot,\cdot)$ and a borelian measure $m$  that satisfies
\begin{itemize}
\item[HS5.] The measure $m$ is a doubling measure, that is, there is $C > 0$  such that 
$$0 < m(B_\rho(x,r)) < C m(B_\rho(x,r/2)) < \infty.$$
for every $x \in I$ and $r > 0$.
\end{itemize}
Note that  if $\rho'$ is power-law equivalent to $\rho$ then  $(I,\rho',m)$ is also a homogeneous space.  By  Macias and Segovia \cite{qsm1}, replacing $\rho$ by an equivalent quasi-metric we can also assume
\begin{itemize}
\item[HS6.] The quasi-balls $B_\rho(x,r)$, with $r> 0$, are open sets. 
\end{itemize}
We will assume that   $m(\{x\})=0$ for every $x \in I$, $I$ is compact and  $m(I)=1$. 

\subsection{Induced Ahlfors regular quasi-metric space.} \label{alf}  For every homogeneous space  as in  the  previous section we can associated an Ahlfors regular quasi-metric space, that is, a triple $(I, m,d,D)$, where $d$ is a quasi-metric, $D> 0$ and 
\begin{equation}\label{alf2}  \frac{1}{\Cll{all}} r^{D}   \leq m(B_{d}(x,r))\leq \Crr{all} r^{D}.\end{equation} 
provided $r\leq r_0$, for some $\Crr{all}, r_0 > 0$.  Indeed $d$ is defined up to power law equivalence, but that will be enough to ours purposes.

Indeed by  Macias and Segovia \cite{qsm1} if we define 
\begin{equation}\label{ms} \beta(x,y)=\inf \{m(B)\colon  \text{  $B$  is a $\rho$-ball that contains $\{x,y\}$ }     \}\end{equation} 
then $\beta$ is a  quasi-metric on $I$ (that it is not necessarily equivalent to $\rho$) such that $(I,\beta,m)$ is a homogeneous space satisfying
\begin{equation}\label{dou}  \frac{1}{\Crr{al}} r\leq m(B_\beta(x,r))\leq \Crr{al} r  \end{equation}
provided $r\leq r_0$, for some $\Cll{al}, r_0 > 0$, so $(m,\beta,1)$ is a Ahlfors regular quasi-metric space. Note that if $\rho$ and $\rho'$ are power-law equivalent then the corresponding quasi-metrics $\beta$ and $\beta'$ are  equivalent. 

One may ask if there is a {\it metric} $d$ and $D > 0$ such that $(m,d,D)$ is an Ahlfors regular {\it metric } space and $d$ is power-law equivalent to $\beta$. Indeed given $\gamma > 0$  such that $(2\Crr{pm})^\gamma\leq  2$ we have that $d_\gamma$ is a metric, where 
\begin{equation} \label{metric_b}d_\gamma(x,y)=\inf \{  \sum_{i=1}^{m-1} \beta(x_i,x_{i+1})^\gamma, \text{ $x_i \in I$, $x_1=x$ and $x_m=y$ } \}.\end{equation} 
By Aimar,  Iaffei and  Nitti \cite{qsm2} and Paluszy\'nski  and Stempak \cite{qsm3} if $(2\Crr{pm})^\gamma\leq  2$ then $d_\gamma$ is a metric on $I$,  $d^{1/\gamma}_\gamma$ is a quasi-metric equivalent to $\beta$. In particular $\beta, d_{\gamma}$ and $d_{\gamma'}$ are power-law equivalent for every $\gamma,\gamma'$. Moreover
$$ \frac{1}{\Crr{alll}} r^{1/\gamma}   \leq m(B_{d_\gamma}(x,r))\leq \Crr{alll} r^{1/\gamma}   $$
for some $\Cll{alll}> 1$ and $r\leq r_2$.   So $(m,d_\gamma, 1/\gamma)$ is a  Ahlfors regular  metric  space. 

\subsection{Good grids in an Ahlfors regular quasi-metric space.}

Due Section \ref{alf} from now on we will consider a general setting of  Ahlfors regular  quasi-metric  space $(m,d,D)$ satisfying (\ref{alf2}) for $r\leq r_0$ and  some $\Crr{all}, r_0 > 0$.

\begin{proposition}[Good grids in Ahlfors regular quasi-metric spaces] \label{fgh} There is a  good  grid $\mathcal{P}=(\mathcal{P}^i)_i$ such that  for every quasi-metric $\alpha$ that  is power-law equivalent to $d$ the following holds: There is $\eta,\Cll{de111}, \Cll{de1111},\Cll{de00},\Cll{oo}\geq 0$ and $\Cll[c]{de000}\in (0,1)$ such that for every $Q\in \mathcal{P}^k$, with $k\geq 1$, there is $z_Q\in Q$ satisfying 
\begin{equation}\label{pp} B_{\alpha}(z_Q,\Crr{de111} \Crr{de000}^k)  \subset Q,\end{equation}
\begin{equation}\label{pp2} diam_\alpha \  Q \leq \Crr{de1111} \Crr{de000}^k  \end{equation}
and
\begin{equation}\label{ppp}  m\{ x\in Q\colon \  \alpha(x,I\setminus Q)\leq \Crr{de00} t  \Crr{de000}^k    \}\leq \Crr{oo} t^{\eta} m(Q).\end{equation} 
\end{proposition}
\begin{proof}
By Christ \cite{christ} there is a  family of partitions $\mathcal{G}^k$, $k \in \mathbb{Z}$, satisfying 
\begin{itemize}
\item[C1.] For every $k$ we have $m(I\setminus \cup_{Q\in \mathcal{G}^k} Q)=0$.
\item[C2.] For every $Q,W \in \mathcal{G}^k$  we have either $Q=W$ or $Q\cap W=\emptyset$.
\item[C3.] For every $Q \in \mathcal{G}^k$ there is an unique $P\in \mathcal{P}^{k-1}$ such that $Q\subset P$.
\item[C4.] For every $Q \in \mathcal{G}^k$ we have $diam_{d} Q \leq \Cll{de} \Cll[c]{de0}^k$.
\item[C5.] For every $Q \in \mathcal{G}^k$ there is $z_Q\in Q$ such that $B_{d}(z_Q,\Cll{de1} \Crr{de0}^k)\subset Q$, with $\Crr{de1}\leq \Crr{de}$.
\item[C6.] For every $Q \in \mathcal{G}^k$ we have 
$$m\{ x\in Q\colon \  d(x,I\setminus Q)\leq t  \Crr{de0}^k    \}\leq \Crr{oo} t^{\eta} m(Q).$$
\end{itemize}
Let $k_0$ be such that 
$$( \Crr{de1}+ \Crr{de} )\Crr{de0}^{k_0} < r_1$$
and
$$\Crr{all}  \Crr{de} \Crr{de0}^{k_0} < 1.$$

Let $K \in \mathbb{N}$ be such that
$$\Crr{maior}= \max \{ \Crr{all} ^{2} \frac{\Crr{de}^D}{\Crr{de1}^D} \Crr{de0}^{DK}, \Crr{all}  \Crr{de}^D\Crr{de0}^{Dk_0}\} < 1.$$
Note that 
$$\Crr{menor}= \min \{ \frac{1}{\Crr{all}^{2}} \frac{\Crr{de1}^D}{\Crr{de}^D} \Crr{de0}^{DK},  \frac{1}{\Crr{all} } \Crr{de1}^D \Crr{de0}^{Dk_0}\} \leq \Crr{maior}.$$
Let $\mathcal{P}^0=\{I\}$ and $\mathcal{P}^i=\mathcal{G}^{k_0+ (i-1)K}$ for every $i\geq 1$. 
Then for every $i\geq 0$, $P\in \mathcal{P}^{i+1}$ and $P\subset Q\in  \mathcal{P}^{i}$ we have
$$\Crr{menor}  \leq  \frac{|P|}{|Q|}\leq \Crr{maior}.$$
It is easy to see that $\mathcal{P}$ has similar properties for every power-law equivalent metric $\alpha$.
\end{proof}

For every grid $\mathcal{P}$ as in Proposition \ref{fgh} we  can consider the Banach space $\mathcal{B}^s_{p,q}(\mathcal{P})$, using $(s,p)$-Souza's atoms, where $p \in [1,\infty)$, $q\in [1,\infty]$ and $ s \in (0, 1/p)$. 

\begin{proposition}\label{indicator}  Let $\mathcal{P}$  be a  grid as in Proposition \ref{fgh} taking $\alpha=d$. Assume that $Dsp< \hat{\eta}$. There is $\Cll{abs3w}$, $\Cll{abs5w}$ such that the following holds.  Let $Q\subset I$ be an open subset such that there is $z_Q\in Q$ satisfying 
\begin{equation}\label{ppe} B_{d}(z_Q,\Cll{de111a} diam_d \  Q )  \subset Q,\end{equation}
and 
\begin{equation}\label{ppe2} m\{ x\in Q\colon \  d(x,I\setminus Q)\leq \Cll{de00x} t  \ diam_d \  Q   \}\leq \Cll{oox} t^{\hat{\eta}} m(Q)\end{equation}
for some $\Crr{de111a}$, $\Crr{de00x}$, $\Crr{oox}$ and  $\hat{\eta} > 0$. Then 
\begin{itemize}
\item[A.] The set  $Q$ is a  $(1-sp, \Cll{abs3}, \Crr{de000}^{\hat{\eta}-Dsp})$-regular domain,  with $$\Crr{abs3}= \Crr{abs3w}  \frac{\Crr{oox}}{\Crr{de00x}^{\hat{\eta}} \Crr{de111a}^D},$$ that is,  there are families $\mathcal{F}^k_Q\subset \mathcal{P}^k$,  such that 
\begin{itemize}
\item[i.] If $P\in \mathcal{F}^k_Q$, $W\in \mathcal{F}^j_Q$ and $P\neq W$ then $P\cap W =\emptyset.$
\item[ii.] We have 
$$Q=\bigcup_{k\geq k_0(Q)} \bigcup_{P \in \mathcal{F}^k_Q} P.$$
\item[iii.]  We have
\begin{align*}  \sum_{P\in \mathcal{F}^k_Q} |P|^{1-sp} 
&\leq \Crr{abs3} \Crr{de000}^{(k-k_0(Q))(\hat{\eta}-Dsp)} |Q|^{1-sp}.
\end{align*} 
\end{itemize}
\item[B.]  We have $1_Q \in \mathcal{B}^s_{p,q}(\mathcal{P})$ for every $q \in [1,\infty]$ and $$|1_Q|_{\mathcal{B}^s_{p,q}(\mathcal{P})}\leq \Crr{abs5w} \big( \frac{\Crr{oox}}{\Crr{de00x}^{\hat{\eta}} \Crr{de111a}^D}\big)^{1/p}  |Q|^{1/p-s}.$$
\end{itemize}
\end{proposition} 
\begin{proof} Note that  
$$\frac{1}{\Crr{all}^{2/D}} \leq \frac{diam_d \ Q}{\Crr{de000}^{k_0(Q)}} \leq  \frac{\Crr{pm}\Crr{de1111} }{\Crr{de111a}\Crr{de000}}.$$

Let $\mathcal{F}^k_Q\subset \mathcal{P}^k$  be the family of all $P\in \mathcal{P}^k$ such that $P \subset Q$ and there is not  $P' \in \mathcal{P}$ satisfying  $P \subsetneqq P'\subset Q$. For every  $P\in \mathcal{F}^k_Q$ define 
$$s_{P,Q}= |P|^{1/p-s}.$$
otherwise set $s_{P,Q}=0$. Note that
\begin{equation}\label{decc} 1_Q= \sum_{P\in \mathcal{P}} s_{P,Q} b_{P}.\end{equation}
Here $b_Q$ is the canonical $(s,p)$-Souza's atom on $Q$. 
If $P\in \mathcal{P}^k$ then 
\begin{equation}\label{pme}  m(P)\geq \frac{1}{\Crr{all}}  \Crr{de111}^D \Crr{de000}^{Dk}\end{equation} 
Note also that if $x \in Q$ and $$d(x,I\setminus Q) > \Crr{de1111} \Crr{de000}^{k-1} $$ then there is $P \in  \mathcal{P}^{k-1}$ such that $x \in P \subset Q$. Consequently if $k \geq k_0(Q)$
\begin{align*} \sum_{P\in \mathcal{F}^k_Q} |P| &= m(\bigcup_{P\in \mathcal{F}^k_Q} P) \\
&\leq  m\{ x\in Q\colon \ d(x,I\setminus Q) \leq \Crr{de1111}  \Crr{de000}^{k-1}    \} \\
&\leq \Crr{oox} \Big( \frac{\Crr{de1111}}{\Crr{de00x}} \frac{\Crr{de000}^{k-1}}{ diam_d \ Q} \Big)^{{\hat{\eta}}} |Q| \\
&\leq \Crr{oox} \big(\frac{\Crr{de1111} \Crr{all}^{2/D}\Crr{de000}^{-1} }{\Crr{de00x}} \big)^{\hat{\eta}}    \Crr{de000}^{(k-k_0(Q)){\hat{\eta}}}|Q|\\
&\leq \Crr{oox}\Crr{all} \big(\frac{\Crr{de1111} \Crr{all}^{2/D}\Crr{de000}^{-1} }{\Crr{de00x}} \big)^{\hat{\eta}} \big(\frac{\Crr{pm}\Crr{de1111} }{\Crr{de111a}\Crr{de000}}\big)^D    \Crr{de000}^{(k-k_0(Q)){\hat{\eta}}}\Crr{de000}^{Dk_0(Q)}\\
&\leq \Cll{abss} \frac{\Crr{oox}}{\Crr{de00x}^{\hat{\eta}} \Crr{de111a}^D}    \Crr{de000}^{(k-k_0(Q)){\hat{\eta}}}\Crr{de000}^{Dk_0(Q)}.
\end{align*} 
for some $\Crr{abss} \geq 0$. It follows from (\ref{ppe}) and  (\ref{pme}) that 
\begin{equation}
\sharp \mathcal{F}^k_Q \leq \Crr{abss} \frac{\Crr{all} }{ \Crr{de111}^D }  \frac{\Crr{oox}}{\Crr{de00x}^{\hat{\eta}} \Crr{de111a}^D}  \Crr{de000}^{(k-k_0(Q))({\hat{\eta}}-D)} 
\end{equation} 
Consequently by (\ref{ppe})
\begin{align*}   \sum_{P\in \mathcal{F}^k_Q}\Big(\frac{|P|}{|Q|}\Big)^{1-sp}
&\leq \Crr{abs3w}  \frac{\Crr{oox}}{\Crr{de00x}^\eta  \Crr{de111a}^D}  \Crr{de000}^{(k-k_0(Q))(\hat{\eta}-D)} \Crr{de000}^{(k-k_0(Q))D(1-sp)} \\
&\leq \Crr{abs3w} \frac{\Crr{oox}}{\Crr{de00x}^\eta \Crr{de111a}^D}\Crr{de000}^{(k-k_0(Q))(\hat{\eta}-Dsp)}.
\end{align*} 
This proves $A$. Together with (\ref{decc}) we have that $A.$ implies $B$.
\end{proof}

\begin{proposition}\label{besovhom}  Let $\mathcal{P}_i$, $i=\star,\circ$, be two grids as in Proposition \ref{fgh} taking $\alpha=d$.  Suppose that
$$\min\{ \eta_\star, \eta_\circ \}>  Dsp.$$
Then $\mathcal{B}^s_{p,q}(\mathcal{P}_\star)=\mathcal{B}^s_{p,q}(\mathcal{P}_\circ)$ and the corresponding norms are equivalent. 
\end{proposition} 
\begin{proof} For each  $\mathcal{P}_i$, $i=\star,\circ$ denote  by $\eta_i, \Crr{de111}^i, \Crr{de1111}^i,\Crr{de00}^i,\Crr{oo}^i\geq 0$ and $\Crr{de000}^i\in (0,1)$ the  corresponding constants in Proposition \ref{fgh}.  Fix $Q\in \mathcal{P}_\star$. Given $P\in \mathcal{P}_\circ$, let $b_{P,Q}=a_P$, where $P$ is  the canonical Souza's atom supported on $P$. Let $Q\in \mathcal{P}_\star^i$.\\

\noindent {\it Claim I.}  We have 
\begin{equation} \label{con} k_i \leq k_0^\circ(Q)\leq k_i+b,\end{equation}
where 
\begin{equation}\label{ki} k_i= i \frac{\ln \Crr{de000}^\star}{\ln \Crr{de000}^\circ} + \frac{\ln \Crr{all}^2 + D(\ln \Crr{de1111}^\star-\ln \Crr{de1111}^\circ)}{D\ln \Crr{de000}^\circ}.\end{equation} 
Indeed we have
\begin{equation}\label{qm}  m(Q)\leq \Crr{all}  (\Crr{de1111}^\star (\Crr{de000}^\star)^i)^D\end{equation} 
and if $P\in \mathcal{P}_\circ^k$ then 
\begin{equation}\label{pm}  m(P)\geq\Crr{all}^{-1}  (\Crr{de1111}^\circ (\Crr{de000}^\circ)^k)^D.\end{equation} 
In particular if  $k < k_i$ and  $P \in \mathcal{P}_\circ^k$ then   $P \not\subset Q$, so $P \not\in \mathcal{F}^k_Q$  and  $s_{P,Q}=0$.  In particular $k_0^{\circ}(Q) \geq k_i$.  On the other hand by (\ref{pp}) 

$$B_{d}(z_Q^\star,\Cll{de11c} \Crr{de111}^\star(\Crr{de000}^\circ)^{k_i})  \subset B_{d}(z_Q^\star,\Crr{de111}^\star (\Crr{de000}^\star)^i)  \subset Q,$$
where
$$\Crr{de11c} =  \exp( -\frac{\ln \Crr{all}^2 +D(\ln \Crr{de1111}^\star - \ln \Crr{de1111}^\circ)}{D }.  $$
Let $b\in \mathbb{N}$ be such that 
$$\Crr{de1111}^\circ (\Crr{de000}^\circ)^{b} \leq \Crr{de11c}.$$
Then by (\ref{pp2}) there is  $P$  satisfying  $z^\star_Q\in  P\in \mathcal{P}_\circ^{k_i+b}$ and $P \subset Q$, so (\ref{con}) holds.\\

\noindent {\it Claim II.}  There is $\Crr{de111a}$ such that 
\begin{equation}\label{ppew} B_{d}(z_Q,\Crr{de111a} diam_d \  Q )  \subset Q.\end{equation}
By this follows from (\ref{pp}) and (\ref{pp2}) taking $\Crr{de111a}= \Crr{de111}^\star/\Crr{de1111}^\star. $\\

\noindent{\it Claim III.}  There are $ \Crr{de00x} ,  \Crr{oox}$ such that 
\begin{equation}\label{ppe24} m\{ x\in Q\colon \  d(x,I\setminus Q)\leq \Crr{de00x} t  \ diam_d \  Q   \}\leq \Crr{oox} t^{\eta_\star} m(Q)\end{equation}
holds.  

Indeed, by Proposition \ref{fgh} we have
$$ m\{ x\in Q\colon \  d(x,I\setminus Q)\leq \Crr{de00}^\star t  (\Crr{de000}^\star)^i    \}\leq \Crr{oo}^\star t^{\eta_\star} m(Q)$$
By (\ref{pp2}) we have $diam_ d \ Q\leq \Crr{de1111}^\star ( \Crr{de000}^\star)^i$, so
\begin{align*}
&m\{ x\in Q\colon \  d(x,I\setminus Q)\leq \frac{\Crr{de00}^\star}{\Crr{de1111}^\star} t  \ diam_d \  Q   \}\\
\leq &m\{ x\in Q\colon \  d(x,I\setminus Q)\leq \Crr{de00}^\star t  (\Crr{de000}^\star)^i    \}\leq \Crr{oo}^\star t^{\eta_\star} m(Q)
\end{align*}
Claim II and III imply that we can apply Proposition \ref{indicator} for $Q$ and consequently $Q$ is  a  $(1-sp, \Crr{abs3}, (\Crr{de000}^\circ)^{\eta_\star-Dsp})$-regular domain with respect the good grid $\mathcal{P}^\circ$, that is, one can find families  $\mathcal{F}^k_Q\subset \mathcal{P}_\circ^k$  as in Proposition \ref{indicator} taking $\mathcal{P}= \mathcal{P}_\circ$. For every  $P\in \mathcal{F}^k_Q$ define 
$$s_{P,Q}= \Big(\frac{|P|}{|Q|}\Big)^{1/p-s}.$$
otherwise set $s_{P,Q}=0$. Note that
$$a_Q= \sum_{P\in \mathcal{P}_\circ} s_{P,Q} a_P$$
and by Claim I
\begin{align*}  \sum_{P\in \mathcal{F}^k_Q} |s_{P,Q}|^p&= \sum_{P\in \mathcal{F}^k_Q}\Big(\frac{|Q|}{|P|}\Big)^{sp-1} \\
&\leq \Crr{abs3} (\Crr{de000}^\circ)^{(k-k_0^\circ(Q))(\eta_\star-Dsp)}\\
&\leq \Cll{abs33} (\Crr{de000}^\circ)^{(k-k_i)(\eta_\star-Dsp)}
\end{align*} 
with  $\Crr{abs33}=  \Crr{abs3}(\Crr{de000}^\circ)^{-b(\eta_\star-Dsp)}$.

By Proposition \ref{besov-trans}.C in \cite{smania-besov}  we have that $\mathcal{B}^s_{p,q}(\mathcal{P}_\star)\subset \mathcal{B}^s_{p,q}(\mathcal{P}_\circ)$ and there exists $C$ such that 
$$|f|_{\mathcal{B}^s_{p,q}(\mathcal{P}_\circ)}\leq C|f|_{\mathcal{B}^s_{p,q}(\mathcal{P}_\star)}$$
for every $f \in \mathcal{B}^s_{p,q}(\mathcal{P}_\star).$ Exchanging the roles of $\mathcal{P}_\star$ and $\mathcal{P}_\circ$ in the above argument we obtain the reverse inclusion and   inequality.  
\end{proof}

\begin{definition}\label{besovhomog}Let $(I,\rho,m)$ be a homogeneous space and $(m,d,1)$ be a Ahlfors regular  quasi-metric  space such that $d$ is power-law equivalent to $\rho$.   Let  $\tilde{\eta}$ be the supremum of all $\eta$ that admits a good grid  as in Proposition \ref{fgh} with $\alpha=d$. Let $s \in (0,1/p)$, $p\in [1,\infty)$ and $q \in [1,\infty]$. Assume $\tilde{\eta} > sp$. We define the  {\it Besov space} $\mathbb{B}^s_{p,q}$ on the homogeneous space $(I,\rho,m)$  as  $\mathcal{B}^s_{p,q}(\mathcal{P})$. The Besov space $\mathbb{B}^s_{p,q}$ is well-defined due Proposition \ref{besovhom}. 
\end{definition}



Han, Lu and Yang \cite{han2}  defined  inhomogeneous Besov spaces $B_{HLY}(s,p,q)$ for homogeneous spaces introducing  a type of Calder\'on reproducing formula.  They also obtained an atomic decomposition of these Besov spaces, that we describe now. Let $\beta$ as in (\ref{ms}) and $\Crr{cd}$ be such that
$$\beta(x,z)\leq \Cll{cd}( \beta(x,y)+\beta(y,z))$$
Note that $(I,m,\beta,1)$ is an Ahlfors regular quasi-metric space.  Let $\mathcal{P}^i$ be a grid as in Proposition \ref{fgh} taking $\alpha=\beta$. Fix $\gamma > 0$  small enough such that $d=d_\gamma$ as defined in (\ref{metric_b}) is a metric satisfying 
\begin{equation}\label{metric_bb}  \frac{1}{\Cr{nn}} d_\gamma (x,y) \leq     \beta^\gamma(x,y)\leq \Cl{nn}  d_\gamma (x,y)\end{equation} 
Let $s > 0$ be  such that $0< s< \gamma$.  A Han-Lu-Yang $\gamma$-block  $a_Q$ associated to $Q\in \mathcal{P}^k$ is a function
$$a_Q\colon I \rightarrow \mathbb{C}$$
such that
\begin{itemize}
\item[i.] $supp \ a_Q \subset B_\beta(z_Q,3\Crr{cd}\Crr{de1111} \Crr{de000}^k ).$
\item[ii.] $|a_Q|_\infty  \leq |Q|^{s-1/p}.$
\item[iii.] $|a_Q(x)-a_Q(y)| \leq |Q|^{s-1/p-\gamma} \beta(x,y)^\gamma.$
\end{itemize}

A function $f \in L^p$ belongs to $B_{HLY}(s,p,q)$ if (Theorem 2.1 and Proposition 3.1 in \cite{han2}, and also Theorem 6.5 and Remark 6.20 in Han and  Sawyer \cite{hs}) we can write
\begin{equation}\label{rp22}  f = \sum_k \sum_{Q\in \mathcal{P}^k} c_Q a_Q,\end{equation} 
where $a_Q$ is a Han-Lu-Yang $\gamma$-block  $a_Q$ associated to $Q$ and the convergence is absolute  in $L^p$, that is 
$$\sum_k |\sum_{Q\in \mathcal{P}^k} c_Q a_Q|_p < \infty$$
and
$$\Big(   \sum_k \big(\sum_{Q\in \mathcal{P}^k}  |c_Q|^p \big)^{q/p} \Big)^{1/q} < \infty.$$
We define 
$$|f|_{B_{HLY}(s,p,q)} =\inf \Big(   \sum_k \big(\sum_{Q\in \mathcal{P}^k}  |c_Q|^p \big)^{q/p} \Big)^{1/q},$$
where the infimum runs over all possible representations (\ref{rp22}).   We are using a slightly different definition  for an $\gamma$-block but  we also modified  the norm definition accordingly to obtain the same Besov space as in Han, Lu and Yang \cite{han2}.

\begin{proposition} \label{iii}
Let $\mathcal{P}$\label{besovreal} be a good grid as  in Proposition \ref{fgh} with $\alpha=\beta$.  Suppose that $\eta >  sp$, $0< s < \gamma$, $p\in [1,\infty)$ and $q\in [0,\infty]$. Then $B_{HLY}(s,p,q)=\mathbb{B}^s_{p,q}(\mathcal{P})$.
\end{proposition}
\begin{proof}  We will prove it in  several steps.\\ \\
\noindent {\it Step I.}  Consider the grid 
$$\mathcal{F}^k=\{ B_\beta(z_Q,3\Crr{cd}\Crr{de1111} \Crr{de000}^k ) \colon \ Q\in \mathcal{P}^k  \}.$$
Let  $\mathcal{A}^{HLY}_{s,p}(B_\beta(z_Q,3\Crr{cd}\Crr{de1111} \Crr{de000}^k ))$ be the class of atoms  that consists of functions $\phi$ that satisfy
\begin{itemize}
\item[i.] $supp \ \phi \subset B_\beta(z_Q,3\Crr{cd}\Crr{de1111} \Crr{de000}^k ).$
\item[ii.] $|\phi|_\infty  \leq |B_\beta(z_Q,3\Crr{cd}\Crr{de1111} \Crr{de000}^k )|^{s-1/p}.$
\item[iii.] $|\phi(x)-\phi(y)| \leq |Q|^{s-1/p-\gamma} \beta(x,y)^\gamma.$
\end{itemize}
Note that  there is $\Cll{www} > 1$ such that for every $Q\in \mathcal{P}$
$$ \frac{1}{\Crr{www}} \leq     \frac{|B_\beta(z_Q,3\Crr{cd}\Crr{de1111} \Crr{de000}^k )|}{|Q|}\leq \Crr{www}$$
This implies that  $B_{HLY}(s,p,q)=\mathcal{B}^s_{p,q}(\mathcal{A}^{HLY}_{s,p})$ and its norms are equivalent.
 We will prove that $\mathcal{B}^s_{p,q}(\mathcal{A}^{HLY}_{s,p})=\mathcal{B}^s_{p,q}$.
\ \\ \\
\noindent {\it Step II.}  Let $e_Q$ be  the canonical $(s,p)$-Souza's  atom on $Q$. We claim that $e_{Q}\in \mathcal{B}^s_{p,q}(\mathcal{A}^{HLY}_{s,p})$ for every $Q\in \mathcal{P}^{k_0}$, $k_0\geq 1$. Indeed for every $k\geq k_0$, let $\mathcal{F}^k(Q)$ be the family of all
 $P\in \mathcal{P}^k$   such that 
\begin{itemize}
\item[A.]  We have 
$$P\subset B_{\beta}(z_P,  3 \Crr{cd}\Crr{de1111} \Crr{de000}^k )\subset Q.$$
\item[B.] If  $P\subset W\in \mathcal{P}^{i}$, with $i < k$ we have 
$$B_{\beta}(z_W, 3\Crr{cd} \Crr{de1111} \Crr{de000}^{i})\not\subset Q.$$
\end{itemize}
There is $\Cll{kk}> 0$ such that for every  $x \in B_{\beta}(z_P,  2 \Crr{de1111} \Crr{de000}^k )$, with $P\in \mathcal{F}^k(Q)$,   we have 
 $$  \frac{1}{\Crr{kk}} \Crr{de000}^k    \leq d_\beta (x,I\setminus  Q) \leq  \Crr{kk} \Crr{de000}^k.$$
In particular by (\ref{ppp}) there is $\Crr{some},\Crr{some2}$ such that 
$$m(\bigcup_{P\in \mathcal{F}^k(Q)} P) \leq \Cll{some} \Crr{de000}^{\eta( k-k_0)} |Q|\leq \Cll{some2} \Crr{de000}^{k_0+\eta( k-k_0)}$$
for every $k$, and consequently
\begin{equation}\label{cardfk}\# \mathcal{F}^k(Q) \leq \Cll{some3} \Crr{de000}^{(\eta-1)( k-k_0)} \end{equation}
for some constant $\Crr{some3}$.
 Note that 
$$Q= \bigcup_{k\geq k_0}  \bigcup_{P\in \mathcal{F}^k(Q)} P = \bigcup_{k\geq k_0}  \bigcup_{P\in \mathcal{F}^k(Q)} B_{\beta}(z_P,  2 \Crr{de1111} \Crr{de000}^k ).$$
For every $P\in \mathcal{F}^k(Q)$ such that $I\setminus B_{\beta}(z_P,  2 \Crr{de1111} \Crr{de000}^k )\neq \emptyset$  define
\begin{equation}\label{pss} \psi_P(x)=|P|^{-\gamma}d_\gamma(x,I\setminus B_{\beta}(z_P,  2 \Crr{de1111} \Crr{de000}^k )),\end{equation}
otherwise  let $\psi_P(x)=1$. There is  $\Cll{qq}> 1 $ such that for every $P\in \mathcal{F}^k(Q)$, $k\geq 1$, we have 
\begin{equation}\label{oi} \psi_P(x)\leq  \Crr{qq}\end{equation}
for every $x\in I$ and
\begin{equation}\label{oi2} \frac{1}{\Crr{qq}} \leq  \psi_P(x)\end{equation}
for every $x\in P$. We claim that 
$$|\psi_P(x)-\psi_P(y)|\leq 2 \Crr{qq}\Crr{nn} |P|^{-\gamma} \beta(x,y)^{\gamma}$$
for every $x,y\in I$. Indeed, if $d_\gamma(x,y)\geq |P|^\gamma$ then
$$|\psi_P(x)-\psi_P(y)|\leq 2 \Crr{qq} \leq 2 \Crr{qq} |P|^{-\gamma}d_\gamma(x,y)\leq 2 \Crr{qq} \Crr{nn} |P|^{-\gamma}\beta(x,y)^{\gamma},$$
and if $d_\gamma(x,y)\leq  |P|^\gamma$ then
$$|\psi_P(x)-\psi_P(y)|\leq |P|^{-\gamma}d_\gamma(x,y)   \leq  \Crr{nn} |P|^{-\gamma}\beta(x,y)^{\gamma}.$$
In particular
$$a_P(x)=\frac{1}{2 \Crr{qq} \Crr{nn} }  |P|^{s-1/p}\psi_P(x).$$
is a $\gamma$-block.  
Let 
$$f_Q(x)=\sum_{k\geq k_0}\sum_{P \in \mathcal{F}^k(Q)} 2 \Crr{qq} \Crr{nn}  |P|^{1/p-s} a_P$$
Note that $f_Q(x)=0$ if $x\not\in Q$ and $f_Q\in \mathcal{B}^s_{p,q}(\mathcal{A}^{HLY}_{s,p})$, since by (\ref{cardfk}) there is $\Cll{some4}, \Cll{uu}$ such that 
\begin{eqnarray}
\sum_{P\in \mathcal{F}^k(Q)}  |P|^{1-sp}  &\leq&   \sum_{P\in \mathcal{F}^k(Q)} (\Crr{de1111} \Crr{de000}^k )^{1-sp} \nonumber \\
&\leq& \Crr{de1111}^{1-sp}  \Crr{some3} \Crr{de000}^{(\eta-1)( k-k_0)} \Crr{de000}^{k (1-sp)}  \nonumber \\
 &\leq& \Crr{de1111}^{1-sp}  \Crr{some3} \Crr{de000}^{(\eta -sp)( k-k_0)} \Crr{de000}^{(1-sp)k_0} \nonumber \\
 \label{hgq}  &\leq& \Crr{uu}\Crr{de1111}^{1-sp}  \Crr{some3} \Crr{de000}^{(\eta -sp)( k-k_0)}|Q|^{1-sp} 
\end{eqnarray}
and consequently
\begin{eqnarray}
\Big(   \sum_{k\geq k_0} \big(\sum_{P\in \mathcal{F}^k(Q)}  |P|^{1-sp} \big)^{q/p} \Big)^{1/q} 
&\leq& \Crr{uu}^{1/p} \Crr{de1111}^{1/p-s} \Crr{some3}^{1/p} \Big(   \sum_{k\geq k_0}   \Crr{de000}^{\frac{q}{p}(\eta -sp)( k-k_0)} \Big)^{1/q}|Q|^{1/p-s}. \nonumber \\
\label{hq}&\leq&   \Crr{some4}  |Q|^{1/p-s}.
\end{eqnarray}
Finally note that due (\ref{dou}) and (\ref{pp}) there is $\Cll{se},\Cll{see}$ such that for every $P\in \mathcal{F}^k(Q)$, $k\geq 1$, we have
$$\# \cup_j \{R\in  \mathcal{F}^j(Q) \ s.t. \  B_{\beta}(z_R,  2 \Crr{de1111} \Crr{de000}^j )\cap B_{\beta}(z_P,  2 \Crr{de1111} \Crr{de000}^k ) \neq \emptyset   \}\leq \Crr{se}$$
and $|j-k|\leq \Crr{see}$ for every $j$ such that there is  $R\in  \mathcal{F}^j(Q)$ satisfying $$B_{\beta}(z_R,  2 \Crr{de1111} \Crr{de000}^j )\cap B_{\beta}(z_P,  2 \Crr{de1111} \Crr{de000}^k ) \neq \emptyset $$ for some $P\in \mathcal{F}^k(Q)$.  As a consequence $f_Q$ satisfies
$$|f_Q(x)-f_Q(y)|\leq \Cll{seee} |P|^{-\gamma}   \beta(x,y)^\gamma$$
for every $x,y \in B_{\beta}(z_P,  2 \Crr{de1111} \Crr{de000}^k )$, with $P\in \mathcal{F}^k$, $k\geq 1$. Moreover (\ref{oi}) and (\ref{oi2}) implies
$$\frac{1}{\Crr{qq}}  \leq f_Q(x) \leq \Crr{qq} \Crr{se}$$
for every $x \in Q$. This implies that 
$$|\frac{1}{f_Q(x)}-\frac{1}{f_Q(y)}|\leq \Cll{sn} |P|^{-\gamma}   \beta(x,y)^\gamma$$
for every $x,y \in B_{\beta}(z_P,  2 \Crr{de1111} \Crr{de000}^k )$, with $P\in \mathcal{F}^k$, $k\geq 1$. Then
$$b_{P,Q}(x)= \frac{\psi_P(x)}{(\Crr{qq}\Crr{sn} +2 \Crr{qq}^2\Crr{nn} )f_Q(x)}.$$
is a $\gamma$-block on $Q$ and
$$1_Q=  \sum_k \sum_{P\in \mathcal{F}^k(Q)} \Cll{yu} |P|^{1/p-s}  b_{P,Q}.$$
where $\Crr{yu}=\Crr{qq}\Crr{sn} +2 \Crr{qq}^2\Crr{nn}.$ So  (\ref{hq}) implies that $1_Q\in  \mathcal{B}^s_{p,q}(\mathcal{A}^{HLY}_{s,p})$. So if $e_Q$ is the canonical $(s,p)$-Souza's  atom on $Q$ then
$$e_Q=  \sum_k \sum_{P\in \mathcal{F}^k(Q)} r_{P,Q} b_{P,Q},$$
where 
$$r_{P,Q}= \Crr{yu} \Big(\frac{|P|}{|Q|}\Big)^{1/p-s}.$$
Due (\ref{hgq})
\begin{equation}\label{cteo} \sum_{P\in \mathcal{F}^k(Q)}  r_{P,Q}^{p}  \leq  \Cll{yui} \Crr{de000}^{(\eta -sp)( k-k_0)}.\end{equation}
In particular $e_Q\in \mathcal{B}^s_{p,q}(\mathcal{A}^{HLY}_{s,p})$.\\  \\
\noindent {\it Step III.} We apply Proposition \ref{besov-trans}.C  in \cite{smania-besov}   with $\mathcal{A}_2=\mathcal{A}^{HLY}_{s,p}$, $\mathcal{A}_1=\mathcal{A}^{sz}_{s,p}$, $\mathcal{G}=\mathcal{P}$ and $\mathcal{W}=\mathcal{F}$, and  $k_i=i$ for every $i$.   We have that assumption (\ref{besov-ad}) in \cite{smania-besov}  follows from (\ref{cteo}). Consequently $$\mathcal{B}^s_{p,q}(\mathcal{A}^{sz}_{s,p})\subset   \mathcal{B}^s_{p,q}(\mathcal{A}^{HLY}_{s,p})$$
and this inclusion is continuous. \\

\noindent {\it Step IV.}  To prove the reverse continuous inclusion, note that there is $\Cll{lim}$ such that for every $Q\in \mathcal{P}^k$, $k\geq 0$ we have
that $\#\Lambda_Q\leq \Crr{lim}$, $\#\Theta_Q\leq \Crr{lim}$, where
$$\Lambda_Q = \{ P\in \mathcal{P}^k\colon \ P\cap B_\beta(z_Q,3\Crr{cd}\Crr{de1111} \Crr{de000}^k ) \neq \emptyset\},$$
and
$$\Theta_Q = \{ P\in \mathcal{P}^k\colon \ Q\cap B_\beta(z_P,3\Crr{cd}\Crr{de1111} \Crr{de000}^k ) \neq \emptyset\}.$$
There is $\Cll{okp} > 0$ such that the following holds. If $\phi_Q\in \mathcal{A}^{HLY}_{s,p}(B_\beta(z_Q,3\Crr{cd}\Crr{de1111} \Crr{de000}^k ))$ then
for every $P\in \Lambda_Q $ we have
$$\Crr{okp} \phi_Q1_P \in \mathcal{A}^{h}_{s,1,p}(P)$$
where $\mathcal{A}^{h}_{s,1,p}$ are as defined in Section \ref{besov-hatoms} in \cite{smania-besov} taking  $d=\beta$, $D=1$ and $\beta=\gamma$. So if 
$$g=\sum_k \sum_{Q\in \mathcal{P}^k} c_Q \phi_Q$$
is a $\mathcal{B}^s_{p,q}(\mathcal{A}^{HLY}_{s,p})$-representation then define
$$m_P= \frac{1}{\Crr{okp}} \sum_{P\in \Lambda_Q} |c_Q|=\frac{1}{\Crr{okp}} \sum_{Q\in \Theta_P} |c_Q|.$$
We have that 
$$\psi_P= \frac{1}{m_P} \sum_{P\in \Lambda_Q} \Crr{okp} c_Q\phi_Q1_P \in \mathcal{A}^{h}_{s,1,p}(P)$$
and
$$g=\sum_k \sum_{P\in \mathcal{P}^k} m_P \psi_P$$
is a $\mathcal{B}^s_{p,q}(\mathcal{A}^{h}_{s,1,p})$-representation of $g$. Indeed
\begin{eqnarray}
\sum_{P\in \mathcal{P}^k} |m_P|^p&=& \frac{1}{\Crr{okp}^p} \sum_{P\in \mathcal{P}^k} (\sum_{Q\in \Theta_P} |c_Q| )^p \nonumber \\
&\leq&\frac{\Crr{lim}^p}{\Crr{okp}^p}    \sum_{P\in \mathcal{P}^k} \sum_{Q\in \Theta_P} |c_Q|^p \nonumber \\
&\leq&\frac{\Crr{lim}^p}{\Crr{okp}^p} \sum_{Q\in \mathcal{P}^k} \sum_{P\in \Lambda_Q} |c_Q|^p \nonumber \\
&\leq& \frac{\Crr{lim}^{p+1}}{\Crr{okp}^p}  \sum_{Q\in \mathcal{P}^k}  |c_Q|^p.\nonumber
\end{eqnarray} 
So $g \in \mathcal{B}^s_{p,q}(\mathcal{A}^{h}_{s,1,p})$ and
$$|g|_{\mathcal{B}^s_{p,q}(\mathcal{A}^{h}_{s,1,p})}\leq \frac{\Crr{lim}^{1+1/p}}{\Crr{okp}} |g|_{\mathcal{B}^s_{p,q}(\mathcal{A}^{HLY}_{s,p})}.$$
Since by Proposition \ref{besov-hold} in \cite{smania-besov} we have that $\mathcal{B}^s_{p,q}(\mathcal{A}^{h}_{s,1,p})=\mathcal{B}^s_{p,q}$ and their norms are equivalent. The proof is complete. \end{proof}

\subsection{The case $[0,1]^D$.} We would like to connect this abstract setting with classical Besov spaces on $\mathbb{R}^D$ considering the  very well-known homogeneous space  $([0,1]^D, m, d)$, where $d$ is the euclidean metric on $\mathbb{R}^D$ and $m$ is the Lebesgue measure on $\mathbb{R}^D$. Then
$(m,d^D,1)$ is a Ahlfors regular quasi-metric space. Note that we can choose $\gamma=1/D$.
Let  $\mathcal{P}^k$ be the partition of $[0,1]^D$ by  $D$-dimensional cubes with $1$-dimensional faces  with length $2^{-k}$ which  are parallels to the coordinate axes.  Then $(\mathcal{P}^k)_k$ is a good grid satisfying the conclusion of Proposition \ref{fgh} when $\alpha=d^D$  and taking $\Crr{de000}=1/2^D$ and $\eta =1/D$.  By Proposition \ref{besovhom} we have that $\mathbb{B}^s_{p,q}$ is well-defined provided $p \in [1,\infty)$, $q\in [1,\infty]$ and $0< s< 1/(Dp)$. 
Proposition \ref{iii}  implies that 
$B_{HLY}(s,p,q)=\mathbb{B}^s_{p,q}$ for every $p\in [1,\infty)$, $q\in [1,\infty]$ and $0< s< 1/(Dp)$.

Consider the  {\it classical} inhomogeneous Besov space $B_{class}(\hat{s},p,q)$ of $\mathbb{R}^D$, with $0< \hat{s}< 1/p$, $p\in [1,\infty)$ and $q\in [1,\infty]$. Let  $$T=\{ f1_{[0,1]^D}\colon \   f \in   B_{class}(sD,p,q) \}.$$  
We  believe that it is well know that $B_{HLY}(s,p,q)=T$, however since we did not  find a reference we provide a proof.  

\begin{proposition} We have $T=\mathbb{B}^s_{p,q}$, for every $s <1/(Dp)$, $p\in [1,\infty)$ and $q \in [1,\infty]$.\end{proposition}
\begin{proof}
Fix $M > 1$. It follows from  Frazier and Jawerth \cite{fj}  that T is the 
set of all functions $f$ that can be written as in  (\ref{rep55}), where $s_Q$ satisfy  (\ref{rep2})
and $a_Q$ are $C^1$ functions in $\mathbb{R}^D$  such that $supp \ a_Q\subset M Q$,
$$|a_Q|\leq |Q|^{\hat{s}/D -1/p}$$
and
$$|\partial_i a_Q|\leq |Q|^{\hat{s}/D -1/p-1/D}$$
for every $i=1,\dots,D.$ In particular
$$|a_Q(x)-a_Q(y)|\leq \Cll{outra}|Q|^{\hat{s}/D -1/p-1/D}|x-y|$$
for some constant $\Crr{outra}$. Due the atomic decomposition of $B_{HLY}(s,p,q)$ by Han, Lu and Yang \cite{han2}  we conclude that $T\subset B_{HLY}(s,p,q)=\mathbb{B}^s_{p,q}$ (recall that $|x~-~y~|~=~\beta(x,y)^\gamma$). 

To show that 
$\mathbb{B}^s_{p,q}\subset T$ we  need to adapt the proof of Proposition \ref{iii},  firstly considering all functions and balls   defined  in the whole $\mathbb{R}^D$ (opposite to just on $[0,1]^D$). If we modify the definition of $\mathcal{A}^{HLY}_{s,p}(B_\beta(z_Q,3\Crr{cd}\Crr{de1111} \Crr{de000}^k ))$ in Step I requiring additionally that all functions in there must be $C^1$ on $\mathbb{R}^D$ then $T=\mathcal{B}^s_{p,q}(\mathcal{A}^{HLY}_{s,p})$. Moreover on Step II. we need to modify the definition of  $\psi_P$ to get a $C^1$ function on $\mathbb{R}^D$. Replace (\ref{pss}) by 

\begin{equation}\psi_P(x)=\frac{1}{2}|P|^{-2\gamma}d_\gamma^2(x,I\setminus B_{\beta}(z_P,  2 \Crr{de1111} \Crr{de000}^k )),\end{equation}
Since $d_\gamma$ is just the euclidean distance and $B_{\beta}(z_P, r)$, with $r> 0$, is an euclidean ball, we have that $\psi_P$ is $C^1$ on $\mathbb{R}^D$  and it satisfies all the estimates in the proof. Then we can carry out Step III to conclude that $\mathbb{B}^s_{p,q}\subset  T$.
\end{proof}

\section{Gu-Taibleson recalibrated martingale Besov spaces} \label{gu}

Let  $I$  be  a compact  Hausdorff  space  and let  $m$ be a borelian measure such that $m(I)=1$. If $\mathcal{P}$ is a good grid on $I$, we can consider the {\it recalibrated martingale Besov spaces} as defined by  Gu and Taibleson \cite{martingale}. In what follows, we just made the obvious adaptations  of their definition to the compact setting. Instead of dealing directly with the good grid $\mathcal{P}$  to define $\mathcal{B}^s_{p,q}(\mathcal{P})$, Gu and Taibleson  define  a {\it recalibration} of the martingale structure as follows. A recalibration of $\mathcal{P}$  is a new good grid $\mathcal{G}$ such that 
$$\cup_k \mathcal{G}^k\subset \cup_k \mathcal{F}^k$$
and satisfying  the following properties. If 
$$\ell_k= \max_{P\in \mathcal{G}^k}|P|, \ m_k =\min_{P\in \mathcal{G}^k}|P|$$
then
$$\ell_{k+1} < \ell_k$$
and  there exist $\Cll[c]{e1}, \Cll[c]{e2},\delta\in (0,1)$ such that 
$$0< \Crr{e1}\leq \frac{\ell_{k+1}}{\ell_{k}}\leq \Crr{e2}< 1$$
 and
$$1\leq \frac{\ell_k}{m_k}\leq \frac{1}{\delta}$$
for every $k$.  By  Gu and Taibleson  \cite{martingale} a recalibration always exists. 

Given $f\in L_1$, define 
$$f_k(x)=\frac{1}{|Q|}\int_Q f \ dm,$$
for every $x\in Q\in \mathcal{G}^k$ and
$$d_kf(x)=f_k(x)-f_{k-1}(x).$$
Of course
$$\int_Q d_k f \ dm=0$$
for every $Q\in \mathcal{G}^{k-1}$.
Note that $f_k$ is exactly the function $f_k$ defined in (\ref{besov-k0}) in \cite{smania-besov}, if we replace $\mathcal{P}^k$ by $\mathcal{G}^k$ there.
We say that that $f$ belongs to the {\bf Gu-Taibleson recalibrated martingale Besov space} $B_{GT}(s,p,q)$ if 
$$|f|_{B_{GT}(s,p,q)}=(\sum_k (\ell_k^{-s} |d_kf|_p)^q)^{1/q}\sim (\sum_k (\ell_k^{-s} |f- f_k|_p)^q)^{1/q} < \infty.$$
Gu and Taibleson also proved that 
$$|f|_{B_{GT}(s,p,q)}\sim (\sum_k (\ell_k^{-s} |f- f_k|_p)^q)^{1/q}.$$
\begin{proposition} We have $B_{GT}(s,p,q)=\mathcal{B}^s_{p,q}(\mathcal{G})$ and its norms are equivalent.
\end{proposition} 
\begin{proof}
By Theorem \ref{besov-alte} in \cite{smania-besov} we have
\begin{eqnarray*}
|f|_{\mathcal{B}^s_{p,q}(\mathcal{G})}&\sim& \Big(  \sum_k \big(  \sum_{Q\in \mathcal{P}^k} |Q|^{-sp} osc_p(f,Q)^p  \big)^{q/p} \Big)^{1/q}\\
&\leq&  \Big(  \sum_k \big(  \sum_{Q\in \mathcal{P}^k} |Q|^{-sp}   \int_Q |f -f_{k-1} |^p \ dm  \big)^{q/p} \Big)^{1/q}\\
&\leq&  \Big(  \sum_k \big(  \sum_{Q\in \mathcal{P}^k}  \ell_k^{-sp} \int_Q |f-f_{k-1}|^p \ dm  \big)^{q/p} \Big)^{1/q}\\
&\leq&  \Big(  \sum_k \big(\ell_k^{-s}  |f-f_{k-1}|_p  \big)^{q} \Big)^{1/q}\sim |f|_{B_{GT}(s,p,q)}.\\
\end{eqnarray*}
On the other hand by (\ref{besov-estphi}) and Theorem \ref{besov-alte} in \cite{smania-besov}
\begin{eqnarray*}
(\sum_k (\ell_k^{-s} |d_kf|_p|)^q)^{1/q}&=&(\sum_k (\sum_{Q\in \mathcal{G}^k} \ell_k^{-sp} \int_Q |d_kf|^p\ dm)^{q/p})^{1/q}\\
&\leq&C (\sum_k (\sum_{Q\in \mathcal{G}^k}|Q|^{-sp} \int_Q |d_kf|^p)^{q/p})^{1/q}\\
&\leq&C (\sum_k (\sum_{Q\in \mathcal{G}^k}|Q|^{-sp} \int_Q |\sum_{S\in \mathcal{H}_Q} d_S\phi_S|^p \ dm)^{q/p})^{1/q}\\
&\leq&  \hat{C} (\sum_k (\sum_{Q\in \mathcal{G}^k}|Q|^{-sp} \int_Q \sum_{S\in \mathcal{H}_Q} |d_S|^p|\phi_S|^p \ dm)^{q/p})^{1/q}\\
&\leq&  \tilde{C} (\sum_k (\sum_{Q\in \mathcal{G}^k}|Q|^{1-sp-p/2}  \sum_{S\in \mathcal{H}_Q} |d_S|^p)^{q/p})^{1/q}\\
&\leq& C' |f|_{\mathcal{B}^s_{p,q}}.
\end{eqnarray*}
\end{proof}

Following Gu and Taibleson, we can endow $I$ with the metric
\begin{equation}\label{narch} \rho(x,y)=\inf \{ |Q|, \ Q\in \mathcal{P} \text{ and } x,y\in Q\}.\end{equation}
Then it is easy to see that $(I,\rho,m)$ is a  1-Ahlfors regular space, so in particular a homogeneous space. Indeed in this case $\rho=\beta=d_1$, where $\beta$ and $d_1$ are as in (\ref{ms}) and (\ref{metric_b}). The following result is, as far as we know,  new, and says that a  Gu-Taibleson recalibrated martingale Besov spaces are particular examples of  Besov spaces on homogenous spaces.

\begin{theorem} \label{mart} For $s > 0$ small and $p,q\geq 1$ the Gu-Taibleson recalibrated martingale Besov space $B_{GT}(s,p,q)$ does not depend on the chosen recalibration and it coincides with the  Besov space $\mathbb{B}^s_{p,q}=\mathcal{B}^s_{p,q}(\mathcal{G})=B_{HLY}(s,p,q)$ of the homogeneous space $(I,\rho,m)$.
\end{theorem}
\begin{proof} If $C > 0$ is large enough and $\lambda \in (0,1)$ is small enough, there exists an increasing sequence $k_i$ such that 
$$\frac{1}{C}\lambda^i\leq     |Q|\leq C\lambda^i$$
for every $Q\in \mathcal{G}^{k_i}$. Define the new  grid (which may not be a good grid) $\hat{\mathcal{W}}$ as $\hat{\mathcal{W}}^{k_i}=\mathcal{G}^{k_i}$ for every $i$ and  $\hat{\mathcal{W}}^{k}=\emptyset$ if $k\not\in \{k_i\}_i.$   One can easily prove using Proposition \ref{besov-trans}  in \cite{smania-besov} that $\mathcal{B}^s_{p,q}(\hat{\mathcal{W}})=\mathcal{B}^s_{p,q}(\mathcal{G})=B_{GT}(s,p,q)$ and theirs norms are equivalent.   Define the good grid $\mathcal{W}$ as $\mathcal{W}^i=\hat{\mathcal{W}}^{k_i}.$  We have $\mathcal{B}^s_{p,q}(\mathcal{W})=\mathcal{B}^s_{p,q}(\hat{\mathcal{W}})$. Note that  $\mathcal{W}$ is a good grid as in Proposition \ref{fgh} taking $\Crr{de000}=\lambda$. Indeed (\ref{pp}) and (\ref{pp2}) are easy to verify.  Note that if $x\in Q\in \mathcal{W}$ and $y\not\in Q$ we have
$$d(x,y)\geq |Q|\geq \frac{1}{C} \lambda^i.$$
and so  (\ref{ppp}) holds and consequently (see Definition \ref{besovhomog}) we have $\mathcal{B}^s_{p,q}(\mathcal{W})=\mathbb{B}^s_{p,q}=B_{HLY}(s,p,q)$.
\end{proof}
We are going to see in Remark  \ref{exotic2} that it is possible to choose the initial good grid $\mathcal{P}$ is such way that $\mathcal{B}^s_{p,q}(\mathcal{P})$ {\it is not } $\mathbb{B}^s_{p,q}(I,\rho,m)$. That means that the families of spaces defined by Souza's atoms is richer than the  families of Besov spaces on homogeneous spaces and recalibrated martingale Besov spaces.

\section{Examples of strongly regular domains }\label{srd}

\begin{figure}
\includegraphics[scale=0.4]{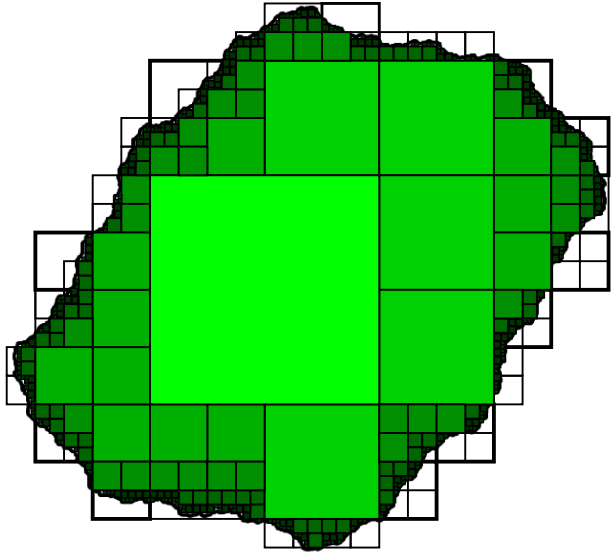}
\caption{A Julia  set $J$ of a quadratic polynomial $x^2+c$, with $c$ small, is a quasicircle. Every quasi-circle delimits a regular domain, but indeed $J$ is the boundary of a strongly regular domain. }
\end{figure}

Let's  assume that  $(I,d, m,D)$ is an Ahlfors quasi-metric space, such that (\ref{alf2}) holds for  every $x\in I$ and $r\leq r_0$. 
The following result give sufficient geometric conditions on $\partial\Omega$ for $\Omega$ to be a strongly regular domain, and in particular for $1_\Omega$ to define a multiplier in appropriated Besov spaces (see Section \ref{srd}). This is obviously  a generalisation of results in Triebel\cite{ns} to our setting.

\begin{proposition}\label{quasi} Let $\mathcal{P}$ be good grid for $(I,d, m,D)$ as in Proposition \ref{fgh}.  Let $K\subset I$ be a closed subset such that $(K,d,\mu,\alpha)$ is Ahlfors regular quasi-metric-space, with $\alpha < D$ and some finite borelian measure $\mu$, that is,
for every $x\in K$ and $r\leq r_1$
$$ \frac{1}{\Cll{ac2}} r^\alpha\leq      \mu(B_d(x,r))\leq \Crr{ac2} r^\alpha.$$

Then  $K^c$ is a $(\alpha/D, \Cll{rp2},\Cll{k1})$-strongly regular domain for some $ \Crr{rp2}\geq 0$ and $\Crr{k1}$ satisfying 
 $$4\Crr{de1111} \Crr{de000}^{\Crr{k1}-1} < \min \{r_0,r_1\},$$ 
 that is, for every $Q\in \mathcal{P}$ such that $k_0(Q)\geq \Crr{k1}$  one can  find  families $\mathcal{F}^k(Q\cap~K^c)~\subset~\mathcal{P}^k$, $k\geq  k_0(Q)$,  such that 
\begin{itemize}
\item[A.] We have $$Q\cap K^c = \cup_{k\geq k_0(Q)} \cup_{P\in \mathcal{F}^k(Q\cap K^c)} P$$
\item[B.] If $P,J \in \cup_{k\geq k_0(Q)} \mathcal{F}^k(Q\cap K^c)$ and $P\neq J$ then $P\cap J=\emptyset$. 
\item[C.] We have
\begin{equation}\label{dom}  \sum_{P\in \mathcal{F}^k(Q\cap K^c)} |P|^{\alpha/D}\leq \Crr{rp2} |Q|^{\alpha/D}.\end{equation} 
\end{itemize}
Suppose additionally that every element of $\mathcal{P}$ is a connected set.  Then every subset $\Omega\subset K^c$  such that $\partial \Omega\subset K$  is also a $(\alpha/D, \Crr{rp2})$-strongly regular domain.
\end{proposition} 
\begin{proof} Suppose $Q\cap K\neq \emptyset,$ otherwise there is nothing to do. 
Define
$$\mathcal{F}^{k_0(Q)}(Q\cap K^c)=\{P\in \mathcal{P}^{k_0}, \ P\subset Q\cap K^c    \}.$$
and for $k > k_0(Q)$
$$\mathcal{F}^{k}(Q\cap K^c)=\{P\in \mathcal{P}^{k}, \ P\subset (Q\cap K^c) \setminus \cup_{k_0\leq i< k} \cup_{P\in \mathcal{F}^{i}(Q\cap K^c)} P  \}.$$
Given $P\in \mathcal{F}^{k}(Q\cap K^c)$, let $J_P\in \mathcal{P}^{k-1}$ be such that $P\subset J_P$. Note that $J_P\cap K\neq \emptyset.$ Choose $y_P\in J_P\cap K.$ Then
\begin{equation}\label{u3} B_{d}(z_P,\Crr{de111} \Crr{de000}^k) \subset P\subset J_P\subset B_d(y_P,\Crr{de1111} \Crr{de000}^{k-1}).\end{equation}
We claim   that for every $Q$ 
\begin{equation}\label{u2}\sum_{P\in \mathcal{F}^{k}(Q\cap K^c)} 1_{B_d(y_P,\Crr{de1111} \Crr{de000}^{k-1})} \leq    \Crr{upper}.\end{equation}
Here
$$\Cll{upper}=  \frac{\Crr{all}^2 (2\Crr{de1111})^D }{\Crr{de111}^D\Crr{de000}^D}.$$
Indeed, suppose  there is $y$ and $P_1, \dots, P_{N}\in  \mathcal{F}^{k}(Q\cap K^c)$ such that $y\in B_d(y_{P_i},\Crr{de1111} \Crr{de000}^{k-1})$, for every $i\leq N$. In particular
$$\cup_i B_{d}(z_{P_i},\Crr{de111} \Crr{de000}^k) \subset \cup_i P_i\subset \cup_i J_{P_i}\subset B_d(y, 2\Crr{de1111} \Crr{de000}^{k-1}).$$
so
$$N \frac{1}{\Crr{all}}  (\Crr{de111} \Crr{de000}^k)^D \leq      \Crr{all} (2\Crr{de1111} \Crr{de000}^{k-1})^D$$
and consequently $N\leq \Crr{upper}.$ This proves the claim.  Choose $q\in Q\cap K$. Then  (\ref{u3})  and (\ref{u2}) implies 
\begin{eqnarray*}
\sum_{P\in \mathcal{F}^{k}(Q\cap K^c)} |P|^{\alpha/D}  &\leq&   \Crr{all} \sum_{P\in \mathcal{F}^{k}(Q\cap K^c)} |\Crr{de1111} \Crr{de000}^{k}|^{ \alpha}\\
&\leq&  \Crr{all}  \Crr{de000}^{\alpha} \sum_{P\in \mathcal{F}^{k}(Q\cap K^c)} |\Crr{de1111} \Crr{de000}^{k-1}|^\alpha \\
&\leq& \Crr{ac2} \Crr{all}  \Crr{de000}^{\alpha}  \sum_{P\in \mathcal{F}^{k}(Q\cap K^c)} \mu(B_d(y_P,\Crr{de1111} \Crr{de000}^{k-1}))\\
&\leq&  \Crr{upper} \Crr{ac2} \Crr{all}  \Crr{de000}^{\alpha}\mu(\bigcup_{P\in \mathcal{F}^{k}(Q\cap K^c)} B_d(y_P,\Crr{de1111} \Crr{de000}^{k-1}))\\
&\leq&  \Crr{upper} \Crr{ac2} \Crr{all}  \Crr{de000}^{\alpha}\mu( B_d(q,4\Crr{de1111} \Crr{de000}^{k_0(Q)-1}))\\
&\leq&  \Crr{upper} \Crr{ac2}^2 \Crr{all}  \Crr{de000}^{\alpha}|4\Crr{de1111} \Crr{de000}^{k_0(Q)-1}|^\alpha\\
&\leq&  4^\alpha \Crr{upper} \Crr{ac2}^2 \Crr{all} \frac{\Crr{de1111}^\alpha}{\Crr{de111}^\alpha}(| \Crr{de111}\Crr{de000}^{k_0(Q)}|^D)^{\alpha/D}\\
&\leq&  4^\alpha \Crr{upper} \Crr{ac2}^2 \Crr{all}^{1+\alpha/D} \frac{\Crr{de1111}^\alpha}{\Crr{de111}^\alpha}(m(B_d(z_Q,\Crr{de111}\Crr{de000}^{k_0(Q)})))^{\alpha/D}\\
&\leq&  4^\alpha \Crr{upper} \Crr{ac2}^2 \Crr{all}^{1+\alpha/D} \frac{\Crr{de1111}^\alpha}{\Crr{de111}^\alpha} |Q|^{\alpha/D}\\
\end{eqnarray*}
Take
$$\Crr{rp2}= 4^\alpha \Crr{upper} \Crr{ac2}^2 \Crr{all}^{1+\alpha/D} \frac{\Crr{de1111}^\alpha}{\Crr{de111}^\alpha} .$$
Suppose now  that every element of $\mathcal{P}$ is a connected set and let  $\Omega\subset K^c$  be such that $\partial \Omega\subset K$. Given $Q\in \mathcal{P}$  let $\mathcal{F}^k(Q\cap\Omega)\subset \mathcal{F}^k(Q\cap K^c)$ be the subfamily of all $P\in \mathcal{F}^k(Q\cap K^c)$ such that $P\cap \Omega\neq \emptyset$. Since $P$ is a connected set, if $P$ is not contained in $\Omega$ we would have $\emptyset\neq P\cap \partial \Omega \subset P\cap K$, which contradicts $P\in \mathcal{F}^k(Q\cap K^c)$. So $P\subset \Omega$ and in particular
$$Q\cap \Omega = \cup_{k\geq k_0(Q)} \cup_{P\in \mathcal{F}^k(Q\cap \Omega)} P.$$
This easily implies that $\Omega$ is also a $(\alpha/D, \Crr{rp2},\Crr{k1})$-strongly regular domain.
\end{proof}

\begin{remark} \label{juliaex} The theory of conformal expanding dynamical systems provides plenty of examples of strongly regular domains on $\mathbb{R}^2$.  For instance consider   the Julia set $J$ of a hyperbolic polynomial. This is a compact subset, so let $I$ be a square such that $J \subset I$. We endowed $I$ with a good grid generated by dyadic squares $\mathcal{D}$ and the  Lebesgue measure $m$. The associated Besov space $\mathcal{B}^s_{p,q}(I,\mathcal{D})$ coincides with the Besov space of the homogenous space $(I,m)$.  Then $J$ supports  a {\it geometric measure } $\mu$ that satisfies the assumptions of Proposition \ref{quasi}.  See Przytycki and Urba\'nski \cite{pu} and also the survey Urba\'nski \cite{u}. In particular every connected component of  $I\cap J^c$ is a strongly regular domain of the homogenous space $(I,m)$. We can consider for instance the Julia set $J$  of a quadratic polynomial $x^2+c$, with $c$ small (see Figure 3.)
\end{remark}

\vspace{1cm}
\newpage

\centerline{ \bf  II. GETTING $\mathbb{R}$EAL.}
\addcontentsline{toc}{chapter}{\bf II. GETTING $\mathbb{R}$EAL.}
\vspace{1cm}

\section{Isometry with Besov spaces defined by intervals}
A case that  provides a very rich class of examples is when   $I=[0,1]$, $m$ is the Lebesgue measure and the good grid consists in sequences of partitions by intervals. It turns out that in the point of view of Banach spaces up to isometries, those Besov spaces are the only ones.

\begin{proposition} Let $\mathcal{B}^s_{p,q}(I, \mathcal{P},\mathcal{A}^{sz}_{s,p})$, with $0< s$ and $p,q\geq 1$,  be a Besov-ish space defined on a  probability space $(I,m)$ with a good grid $\mathcal{P}$. Then there is a good grid $\mathcal{G}$ formed by partitions by intervals of $[0,1]$ such that, considering the Lebesgue measure on $I$ we have that the  Banach space $\mathcal{B}^s_{p,q}(I, \mathcal{P},\mathcal{A}^{sz}_{s,p})$ is isometric to $\mathcal{B}^s_{p,q}([0,1], \mathcal{G},\mathcal{A}^{sz}_{s,p})$.
\end{proposition}
\begin{proof} We can easy define a function $h$ that associated to each $P\in \mathcal{P}$ an  interval $h(P)\subset [0,1]$ such that 
\begin{itemize}
\item[i.] $m(P)=|h(P)|,$
\item[ii.] If $P,Q\in \mathcal{P}$, with $m(P\cap Q)=0$ then $|h(P)\cap h(Q)|=0$.
\item[iii.] If $P,Q\in \mathcal{P}$ with $P\subset Q$ then $\overline{h(P)}\subset \overline{h(Q)}$.
\end{itemize}
Define $\mathcal{G}^k=\{h(P), \ P\in \mathcal{P}^k\}$ and
$$T(\sum_{i\leq n} c_i 1_{P_i})=\sum_{i\leq n} c_i 1_{h(P_i)}$$
for every $c_i\in\mathbb{C}$ and $P_i\in \mathcal{P}$. Then $T$ is well defined and it extends to an isometry
$$T\colon L^p(I)\rightarrow L^p([0,1]).$$
It is easy to see that for every canonical Souza's atom $a_P$ we have $T(a_P)=a_{h(P)}$.  Together with Proposition \ref{besov-lp} in \cite{smania-besov} this implies that 
$$T\colon \mathcal{B}^s_{p,q}(I, \mathcal{P},\mathcal{A}^{sz}_{s,p})\rightarrow \mathcal{B}^s_{p,q}([0,1], \mathcal{G},\mathcal{A}^{sz}_{s,p})$$
is well defined and an isometry.
\end{proof}

\begin{corollary} Let $(I,\rho,m)$ be a homogeneous space such that $I$ is compact and $m(I)=1$. Then the Besov space $B^s_{p,q}$, with $0< s< 1/p$, $p,q\geq 1$ and $s$ small,  is isomorphic  to the corresponding Besov space of the homogeneous space given by $[0,1]$, the usual metric and the  Lebesgue measure.\end{corollary} 
\begin{proof} It follows from Proposition \ref{besovhom} and Proposition \ref{besovreal}.
\end{proof}

\section{Quasisymmetric grids}

In this section   $I=[0,1]$, $m$ is the Lebesgue measure and all  grids consists in sequences of partitions by intervals. We say that a good grid $\mathcal{P}$ consisting on  sequence of partitions of $I$ by intervals  is a  $\Cll{qs}$-{\bf quasisymmetric grid} if  for every $k\geq 0$ and $P, Q\in \mathcal{P}^k$ satisfying $\overline{P}\cap \overline{Q}\neq\emptyset$  we have
\begin{equation}\label{qsc} \frac{1}{\Crr{qs}}\leq \frac{|Q|}{|P|}\leq \Crr{qs}.\end{equation}
Let $Q$ be a closed interval in  $[0,1]$. Define 
$$k_0(Q)=\min \{ k\in \mathbb{N}, \text{ there is $P\in \mathcal{P}^k$ s.t. } P\subset Q\}.$$

\begin{lemma} \label{compi} The following holds.
\begin{itemize}
\item[A.] Let $\mathcal{P}$ be a $(\Crr{menor},\Crr{maior})$-good  grid. Then for every interval $Q=[a,b]\subset [0,1]$ there are families of  intervals $\mathcal{F}_1\subset  \mathcal{F}_2\subset \mathcal{P}^{k_0(Q)}$ such that 
$$\cup_{P\in \mathcal{F}_i} \overline{P}$$
is an interval. Moreover
$$\mathcal{F}_1=\{ P\in \mathcal{P}^{k_0(Q)}, \ \overline{P}\subset \overline{Q} \}$$
and
$$1\leq \#\mathcal{F}_i \leq \frac{2}{\Crr{menor}}+2$$
for $i=1,2$, and
$$\bigcup_{P\in \mathcal{F}_1} \overline{P} \subset \overline{Q}\subset \bigcup_{P\in \mathcal{F}_2} \overline{P}.$$
\item[B.] Suppose additionally that $\mathcal{P}$ is a $\Crr{qs}$-quasisymmetric grid. 
There is $\Cll{th} > 1$ such that for every interval $Q=[a,b]\subset [0,1]$ and every $z\in \partial Q$ and $W \in \mathcal{P}^{k_0(Q)}$ such that $z \in \overline{W}$  we have 
$$\frac{1}{\Crr{th}}\leq \frac{|Q|}{|W|}\leq \Crr{th}.$$
and for every $P\in \mathcal{F}_2$ we have
$$\frac{1}{\Crr{th}}\leq \frac{|Q|}{|P|}\leq \Crr{th}.$$
\end{itemize}
\end{lemma}
\begin{proof}  Let 
$$\mathcal{F}_1=\{ P\in \mathcal{P}^{k_0(Q)}, \ \overline{P}\subset \overline{Q} \}$$
and
$$[a_1,b_1]=\bigcup_{P\in \mathcal{F}_1} \overline{P}.$$
Then
$$1\leq \#\mathcal{F}_1\leq \frac{2}{\Crr{menor}}.$$
If $a_1=0$ then $a=0$ and we define $a_2=a_1$. Otherwise let $a_2< a_1$  be such that  $[a_2,a_1]\in \mathcal{P}^{j(Q)}$. If $b_1=1$ then $b=1$ and we define $b_2=b_1$. Otherwise let $b_2$ be such that $b_1< b_2$ and  $[b_1,b_2]\in \mathcal{P}^{j(Q)}$.  Of course $a\in [a_2,a_1]$ and $b\in [b_1,b_2]$. Let 
$$\mathcal{F}_2=\mathcal{F}_1\cup\{ [a_2,a_1], [b_1,b_2]\},$$
with obvious adaptation if either $a_1=0$ or $b_1=1$. Furthermore if $\mathcal{P}$ is a $\Crr{qs}$-quasisymmetric grid then  for every 
$z\in \{a,b\}$ and $W \in \mathcal{P}^{j(Q)}$ such that $z \in \overline{W}$ we have
$$\Crr{qs}^{-\frac{2}{\Crr{menor}}-2 } (\frac{2}{\Crr{menor}}+2)^{-1}  \leq \frac{1}{|W|}\min \{|P|, \ P\in \mathcal{F}_1 \} \leq \frac{|Q|}{|W|}\leq  \frac{1}{|W|}\sum_{P\in \mathcal{F}_2} |P|\leq  (\frac{2}{\Crr{menor}}+2)  \Crr{qs}^{\frac{2}{\Crr{menor}}+2 },$$
and the same estimates holds replacing $W$ by $P\in \mathcal{F}_2$.
\end{proof}

\begin{proposition} The following holds.
\begin{itemize}
\item[A.] Let $h\colon [0,1]\rightarrow [0,1]$ be a quasisymmetric map and $\mathcal{D}_N$  be the grid of N-adic intervals. Then the grid $\mathcal{P}=h(\mathcal{D}_N)$ defined by 
$$\mathcal{P}^k=\{h(P), \ P\in \mathcal{D}^k_N\}$$
is a quasisymmetric grid. 
\item[B.]  On the other hand if $\mathcal{P}$ is a quasisymmetric grid such that every $P\in \mathcal{P}$ has exactly $N$ children then there is a quasisymmetric function $h$ such that  $\mathcal{P}=h(\mathcal{D}_N)$.
\end{itemize}
\end{proposition}
\begin{proof}[Proof of A] If $h$ is a quasisymmetric map then there is $\Crr{qs} > 0$ such that (\ref{qsc}) holds for every two intervals $P=[x-\delta,x]$, $Q=[x,x+\delta]$ with $\{x-\delta,x+\delta\} \subset [0,1]$.  In particular  $h(\mathcal{D}_N)$ is a quasisymmetric grid.
\end{proof}
\begin{proof}[Proof of B]  Define $\phi\colon [0,1]\rightarrow [0,1]$ as 
$$\phi(x)= \lim_k \frac{1}{N^k} \# \{ P\in \mathcal{P}^k, \ P\subset [0,x]  \}.$$ 
the function $\phi$ is well-defined since the sequence in the r.h.s. is non-decreasing and bounded by $1$. It is easy to see that $\phi$ is monotone increasing and continuous, so it is a homeomorphism. 
Let $L=[x-\delta,x]$, $R=[x,x+\delta]$ with $\{x-\delta,x+\delta\} \subset [0,1]$.  It follows from Lemma \ref{compi}.A. that 
\begin{equation}\label{qw1}\frac{1}{N^{j(L)}} \leq \phi(x)-\phi(x-\delta)\leq \frac{1}{N^{j(L)}} \big( \frac{2}{\Crr{menor}}+2\big)\end{equation}
and
\begin{equation}\label{qw2}\frac{1}{N^{j(R)}} \leq \phi(x+\delta)-\phi(x)\leq \frac{1}{N^{j(R)}} \big( \frac{2}{\Crr{menor}}+2\big).\end{equation}
Let $W_L\in \mathcal{P}^{j(L)}$, $W_R\in \mathcal{P}^{j(R)}$ be such that $x\in \overline{W}_R\cap \overline{W}_L$.  Due Lemma \ref{compi}.B we have
$$\Crr{th}^{-1} \leq   \frac{\delta}{|W_R|}\leq \Crr{th}, \ \Crr{th}^{-1} \leq   \frac{\delta}{|W_L|}\leq \Crr{th},$$
so
$$\Crr{th}^{-2} \leq  \min \{  \frac{|W_L|}{|W_R|}, \frac{|W_R|}{|W_L|} \}.$$
On the other hand since $\mathcal{P}$ is a quasisymmetric (and in particular good) grid we have 
$$\min \{  \frac{|W_L|}{|W_R|}, \frac{|W_R|}{|W_L|} \} \leq \Crr{qs} \Crr{maior}^{|j(L)-j(R)|}$$
so $|j(L)-j(R)|\leq \Cll{uuq}$, where $\Crr{uuq}$ does not depend on $x$ and $\delta$.  By (\ref{qw1}) and (\ref{qw2})
$$  N^{-\Crr{uuq}} \big( \frac{2}{\Crr{menor}}+2\big)^{-1} \leq \frac{|\phi(x)-\phi(x-\delta)|}{|\phi(x+\delta)-\phi(x)|}\leq N^{\Crr{uuq}} \big( \frac{2}{\Crr{menor}}+2\big),$$
so $\phi$  is a quasisymmetric map and consequently the same holds for  $h=\phi^{-1}$.
\end{proof}

\begin{proposition} \label{fam} Let $\mathcal{P}$ be a quasisymmetric $(\Cll[c]{menor},\Cll[c]{maior})$-good  grid and $Q\subset [0,1]$ be an interval. There are families of intervals  $\mathcal{F}^k(Q)$, $k\in \mathbb{Z}$,   and increasing sequences $j_i^+,j_i^- \in \mathbb{N}$ such that 
\begin{itemize}
\item[A.] We have that 
$$\cup_{k \in \mathbb{Z}} \mathcal{F}^k(Q)$$
is a countable partition of $Q$.
\item[B.] We have
$$\mathcal{F}^i(Q) \subset  \mathcal{P}^{j_{i}^+}, \ \mathcal{F}^{-i}(Q) \subset  \mathcal{P}^{j_{i}^-},$$
with $$j_0^+=j_0^-=\min \{j\geq 0 \text{ s.t. there is } P \in \mathcal{P}^j \text{ satisfying }  \overline{P}\subset Q\}. $$
\item[C.]  There exists $\theta \in [0,1)$ and $r_0$  such that 
\begin{equation}\label{grow1} |S|\leq \Crr{qs}^{2/\Crr{menor}+1} \Crr{maior}^r    |W|\end{equation}
for every $S\in \mathcal{F}^{-i-r}(Q)$, $W\in \mathcal{F}^{-i}(Q)$, $i\geq 0$,  and 
for every $S\in \mathcal{F}^{i+r}(Q)$, $W\in \mathcal{F}^{i}(Q)$, $i\geq 0$.
\item[D.] We have 
$$1\leq \# \mathcal{F}^k(Q) \leq \frac{2}{\Crr{menor}}.$$
\end{itemize}
\end{proposition}
\begin{proof}  Let $Q=[a,b]$ and 
$$j_0=\min \{j\geq 0 \text{ s.t. there is } P \in \mathcal{P}^j \text{ satisfying }  \overline{P}\subset Q\}. $$
We define families $\mathcal{F}^k_{\star}(Q)$ with $k\in \mathbb{Z}$, in the following way. Let
$$\mathcal{F}^0(Q)= \{ P \in \mathcal{P}^{j_0} \text{ satisfying }  \overline{P}\subset Q \}$$
and 
$$[a_0,b_0]=\overline{\cup_{P\in \mathcal{F}^0(Q)} P}.$$
Let $j^+_0=j^-_0=j_0$. By induction on $i\geq 0$  define $j_i^-$ as the smallest $j >j_{i-1}^-$ such that 
$$\mathcal{F}^{-i}(Q)= \{ P \in \mathcal{P}^{j} \text{ satisfying }  \overline{P}\subset [a,a_{i-1}] \}\neq \emptyset,$$
and  $j_i^+$ as the smallest $j >j_{i-1}^+$ such that
$$\mathcal{F}^{i}(Q)= \{ P \in \mathcal{P}^{j} \text{ satisfying }  \overline{P}\subset [b_{i-1},b] \}\neq \emptyset$$
and 
$$[a_i,b_i]=(\cup_{0\leq k\leq i} \cup_{P\in \mathcal{F}^k(Q)} \overline{P})\  \bigcup \ (\cup_{0\leq k\leq i} \cup_{P\in \mathcal{F}^{-k}(Q)} \overline{P}) .$$
Note that
$$1\leq \# \mathcal{F}^k(Q) \leq \frac{2}{\Crr{menor}}$$
for every $k\in \mathbb{Z}$ and
$$\cup_{k\in \mathbb{Z}} \mathcal{F}^k(Q)$$
is a countable partition of $Q$.

Let $P\in \mathcal{F}^{-i}(Q)$  be such that $a_i\in \partial P$.  Let $P'\in \mathcal{P}^{j_i^-}$ be such that $P'\subset [0,a_i]$ and $\overline{P'}\cap P\neq \emptyset$. Since $P'\not\in  \mathcal{F}^{-i}(Q)$ we have that 
\begin{equation}\label{incl} [a,a_i]\subset P'\end{equation}
Moreover since $\mathcal{P}$ is a quasymmetric grid 
we have  $|P'|\leq \Crr{qs} |P|$ and $|P'|\leq \Crr{qs}^{2/\Crr{menor}+1} |W|$ for every $W\in  \mathcal{F}^{-i}(Q)$. In particular 
\begin{equation}\label{vv} |a_{i+1}-a_i|\leq |a-a_i|\leq \Crr{qs} |a_i-a_{i-1}|,\end{equation}
and for every $S\in \mathcal{P}^{j_i^++r}$ such that $S\subset  [a,a_i]$ we have
$$|S|\leq \Crr{maior}^r \Crr{qs}^{2/\Crr{menor}+1} |W|$$
for every $W\in  \mathcal{F}^{-i}(Q)$. 
We can use the same argument replacing $a$ by $b$  and $a_i$ by $b_i$ in the above argument. So  C. holds.
\end{proof}

\section{Exotic $\mathcal{B}^s_{p,q}$ for $p\neq q$} \label{exotics}

The goal of this section is to show how sensitive is the dependence of $\mathcal{B}^s_{p,q}$, for $p\neq q$, with respect to the grid $\mathcal{P}$. Choosing distinct grids in a familiar space as $([0,1],m)$, where $m$ is the Lebesgue measure, may   give origin to distinct Besov-ish spaces. The reader should compare this result with Bourdaud and Sickel \cite{change1} and  Bourdaud \cite{change2}

\begin{proposition}\label{inc} Let  $I=[0,1]$, $m$ is the Lebesgue measure, and $\mathcal{P}_\star$ and $\mathcal{P}_\circ$ be quasisymmetric grids  such that for every $P\in \mathcal{P}_i^k$, $i=\star,\circ$ and $k\geq 0$ we have
$$\# \{ Q\in  \mathcal{P}_i^{k+1} \ s.t. \ Q\subset P \}=2.$$
Let $h_i\colon [0,1]\rightarrow [0,1]$ be a quasisymmetric map such that $h_i(\mathcal{P}_i)=\mathcal{D}$.  The following statements are equivalent. 
\begin{itemize}
\item[A.]  We have $\mathcal{B}^s_{p,q}(\mathcal{P}_\circ) \subset  \mathcal{B}^s_{p,q}(\mathcal{P}_\star)$. 
\item[B.] We have  $\mathcal{B}^s_{p,q}(\mathcal{P}_\circ)= \mathcal{B}^s_{p,q}(\mathcal{P}_\star)$.
\item[C.]  The map  $ h_\star \circ h_\circ^{-1}$ is a bi-Lipchitz function.
\end{itemize}
\end{proposition}
\begin{proof} Without loss of generality we assume that $\mathcal{P}_\star$ and $\mathcal{P}_\circ$ are $(\Crr{menor},\Crr{maior})$-good grids. For an interval $Q\subset [0,1]$ define
$$j_0^\star(Q)=\min \{j\geq 0 \text{ s.t. there is } P \in \mathcal{P}^j_\star \text{ satisfying }  \overline{P}\subset Q\}. $$
$$j_0^\circ(Q)=\min \{j\geq 0 \text{ s.t. there is } P \in \mathcal{P}^j_\circ \text{ satisfying }  \overline{P}\subset Q\}. $$
$$j_0^D(Q)=\min \{j\geq 0 \text{ s.t. there is } P \in \mathcal{D}^j_\circ \text{ satisfying }  \overline{P}\subset Q\}. $$
We claim that

\noindent {\it Claim 1.} There is $\Cll{cj}\geq 1$ such that for every $Q\in \mathcal{P}_\circ^k$ and $P\in \mathcal{P}_\star^{j_0^\star(Q)}$ satisfying $P\subset Q$ we have $k\leq j_0^\circ(P)\leq k+\Crr{cj}$. \\

\noindent and\\

\noindent {\it Claim 2.} There is $\Cll{step}\geq 0$ such that for every $Q_1,Q_2 \in \mathcal{P}_\circ^k$, with $k\in \mathbb{N}$, and  satisfying  $\overline{Q}_1\cap  \overline{Q}_2\neq \emptyset$ we have 
$$|j_0^\star(Q_1)-j_0^\star(Q_2)|\leq \Crr{step}.$$

\noindent Both Claims 1 and 2 follows easily from Lemma \ref{compi}. \\

\noindent {\it Claim 3.} Suppose that 
\begin{equation} \label{bou} \Cll{wee}=\sup_k \# \{j^\star_0(Q), Q\in \mathcal{P}_\circ^k     \}< \infty.\end{equation}
Then $ h_\star^{-1}\circ h_\circ$ is a bi-Lipchitz function.\\

Indeed if 
$$M_k= \max  \{j^\star_0(Q), Q\in \mathcal{P}_\circ^k     \} \ and \  m_k=\min  \{j^\star_0(Q), Q\in \mathcal{P}_\circ^k     \}$$
then Claim 2. implies that
$$\sup_k (M_k-m_k) \leq \Crr{wee}(\Crr{step}+1).$$
For every $Q\in \mathcal{P}_\circ$ choose $P_Q\in \mathcal{P}^{j_\star(Q)}$ such that $P_Q\subset Q$. Of course 
$$\#\{P_Q, \ Q\in \mathcal{P}_\circ^k\}=2^k$$
and
$$\{P_Q, \ Q\in \mathcal{P}_\circ^k\} \subset \cup_{m_k\leq j\leq M_k} \mathcal{P}_\star^{j},$$
so
$$2^k\leq \sum_{m_k\leq j\leq M_k} 2^j \leq 2^{M_k+1}\leq 2^{\Crr{step}+1}2^{m_k}.$$
By Lemma \ref{compi}.A each $Q\in \mathcal{P}_\circ^k$ intersects at most $1/\Crr{menor}+2$ intervals in $\mathcal{P}_\star^{j_\star(Q)}$, so $Q$  intersects at most $2^{\Crr{step}}(1/\Crr{menor}+2)$ intervals in $\mathcal{P}_\star^{M_k}$. Since $\mathcal{P}_\circ^k$ covers $[0,1]$ we have
$$\mathcal{P}_\star^{M_k} =\bigcup_{Q\in \mathcal{P}_\circ^k} \{ P\in \mathcal{P}_\star^{M_k}\\ s.t. \  P\cap Q\neq \emptyset \},$$
so
$$2^{m_k}\leq 2^{M_k}\leq  2^{\Crr{step}}(1/\Crr{menor}+2)2^k, $$
so there are constants $A, B$ such that
$$ k+A\leq m_k\leq k +B.$$
Note that for every interval $W\subset [0,1]$ with $k=j_0^D(W)$ we have
$$2^{-k}\leq     |W|\leq 2^{-k+1},$$
so  $k=j_0^\circ(h_\circ^{-1}W))$ and consequently  $j_0^D(h_\star\circ h_\circ^{-1}(W))=j_0^\star(h_\circ^{-1}(W))\in[m_k,M_k].$  We conclude that 
$$ \frac{1}{\Cll{su}} |W|   \leq  2^{-k-B+\Crr{wee}(\Crr{step}+1)}\leq 2^{-M_k}  \leq |h_\star\circ h_\circ^{-1}(W)|\leq 2^{-m_k+1} \leq 2^{-k-A+1}\leq \Crr{su}  |W|,$$
so $h_\star\circ h_\circ^{-1}$ is bi-Lipchitz. This proves Claim 3.\\

\noindent {\it Claim 4.} If $h_\star\circ h_\circ^{-1}$ is a bi-Lipchitz function then  $\mathcal{B}^s_{p,q}(\mathcal{P}_\circ)= \mathcal{B}^s_{p,q}(\mathcal{P}_\star)$.\\

Let  $Q\in \mathcal{P}^i_\circ$. Then 
$$2^{-i}\leq |h_\circ(Q)|\leq 2^{-i+1}$$
and consequently
$$\frac{1}{C} 2^{-i}\leq |h_\star(Q)|\leq C 2^{-i+1},$$
so
$$i+A\leq    j_0^\star(Q)\leq i +B$$
for some $A, B$ that  does not depend on $i$. Let $k_i=i+B.$ Consider Proposition \ref{fam} taking $\mathcal{P}=\mathcal{P}_\star$ and let 
$$\mathcal{G}^k(Q)= \{ P\in \mathcal{F}^{\ell}(Q), \ \ell\geq 0, j^+_\ell=k\}\cup  \{ P\in \mathcal{F}^{\ell}(Q), \ \ell<  0,  j^-_{|\ell|}=k\}.$$
Then by Lemma \ref{compi}.B  we have
$$a_Q= \sum_{k \geq j_0^\star(Q)}  \sum_{P\in \mathcal{G}^k(Q)}    \Big(\frac{|P|}{|Q|}\Big)^{1/p-s} a_P.$$
with 
$$\sum_{\mathcal{G}^k(Q)}    \Big(\frac{|P|}{|Q|}\Big)^{1-sp}\leq   \frac{4 \Crr{th}^{1-sp} }{\Crr{menor}}  \Crr{maior}^{(k-k_i)(1-sp)}.$$
By Proposition \ref{besov-trans}.C in \cite{smania-besov} we have $\mathcal{B}^s_{p,q}(\mathcal{P}_\circ)\subset \mathcal{B}^s_{p,q}(\mathcal{P}_\star)$. The proof of the reverse inclusion is similar.\\

\noindent {\it Claim 5.} Suppose that $p> q$ and 
$$\sup_k \# \{j^\star_0(Q), Q\in \mathcal{P}_\circ^k     \}=+\infty.$$
Then  there is  $f\in \mathcal{B}^s_{p,q}(\mathcal{P}_\circ) \setminus \mathcal{B}^s_{p,q}(\mathcal{P}_\star)$.\\ \\ 

\noindent {\it Step I.} Indeed this implies that  for every $n\in \mathbb{N}$  we can find $r_n, v_n \in \mathbb{N}$ and families with $r_n$ elements  $\{Q^n_1,\dots, Q^n_{r_n}\}\subset \mathcal{P}_\circ^{v_n}$, with $\lim_n r_n=+\infty$ and $$|j^\star_0(Q^n_i)-j^\star_0(Q^n_\ell)|\geq 3\Crr{cj}$$ for every $i\neq \ell$.  Taking a subsequence we can also assume that $v_n$ is increasing, $v_n\geq r_n$,  $$j^\star_0(Q^n_i) + 3\Crr{cj}<  j^\star_0(Q^m_\ell)$$ for every $n<  m$, and 
\begin{equation}\label{ccco1} \sum_{1\leq i\leq r_n}\frac{1}{i}  > 2^{nq}.\end{equation}
For every $n$ and $1\leq i\leq r_n$ choose  $P^n_i\in \mathcal{P}_\star^{j_0^\star(Q^n_i)}$ such that $P^n_i\subset Q^n_i$. Note that Claim 1 implies 
$v_n\leq j_0^\circ(P^n_i)\leq v_n+\Crr{cj}.$\\

\noindent {\it Step II.} For every $n$ and $i\leq r_n$ let $A^n_i,B^n_i$ be the only children of $P^n_i$. Define $S^n_i=(A^n_i,B^n_i)$ and let $\phi_{S^n_i}$ be  the corresponding element of the {\it unconditional basis}  in $L^p$ (note that in Claim 5 we have  $p > 1$) defined in Section \ref{besov-haar} in \cite{smania-besov} using the grid $\mathcal{P}_\star$.  Then
\begin{equation}\label{order}f= \sum_n \sum_{i\leq r_n} \frac{1}{2^ni^{1/q}|P^n_i|^{1/p-s-1/2}}\phi_{S^n_i}.
\end{equation}
is the Haar representation of a function in $L^p$. Indeed we have  $P^n_i\cap P^n_j=\emptyset$ for $i\neq j$, so 
\begin{eqnarray*}
&& |\sum_{i\leq r_n} \frac{1}{2^ni^{1/q}|P^n_i|^{1/p-s-1/2}}\phi_{S^n_i}|_p^p\\
&=& \frac{1}{2^{np}} \sum_{i\leq r_n} | \frac{1}{i^{1/q}|P^n_i|^{1/p-s-1/2}}\phi_{S^n_i}|_p^p\\\
&=&\frac{1}{2^{np}} \sum_{i\leq r_n} \frac{|P^n_i|^{sp}}{i^{p/q}}\leq \frac{1}{2^{np}} \sum_{i=1}^\infty \frac{1}{i^{p/q}}\leq \frac{\Cll{iii}}{2^{np}}.
\end{eqnarray*}
So
\begin{eqnarray*}
&&|\sum_j \sum_{j_0^\star(Q^n_i)=j}     \frac{1}{2^ni^{1/q}|P^n_i|^{1/p-s-1/2}}\phi_{S^n_i}|_p \\&=&  |\sum_n \sum_{i\leq r_n} \frac{1}{2^ni^{1/q}|P^n_i|^{1/p-s-1/2}}\phi_{S^n_i}|_p\\
&\leq& 2\Crr{iii}^{1/p}.
\end{eqnarray*}
We conclude that  $f\in L^p$, with $p > q\geq 1$.\\
 
 \noindent {\it Step III.} Since, the ordering of the sum  in (\ref{order}) does not matter, so
 
 $$f= \sum_{j} \sum_{j_0^\star(Q^n_i)=j}   \frac{1}{2^ni^{1/q}|P^n_i|^{1/p-s-1/2}}\phi_{S^n_i}.$$
 
Since  $j_0^\star(P^n_i)\neq j_0^\star(P^m_\ell)$ provided $(n,i)\neq (m,\ell)$ we have 
\begin{eqnarray*}
 N_{haar}^\star(f)&=&\Big( \sum_{j\geq 1}  \big(  \sum_{P^n_i\in \mathcal{P}^{j}_\star}    \frac{1}{2^{np}i^{p/q}} \big)^{q/p} \Big)^{1/q},\\
 &=&\Big( \sum_{n}\sum_{i \leq r_n}    \frac{1}{2^{nq}i} \Big)^{1/q}\\
&=&\Big( \sum_{n}  \frac{1}{2^{nq}} \sum_{i \leq r_n}    \frac{1}{i} \Big)^{1/q}=+\infty,
 \end{eqnarray*}
 so $f \notin \mathcal{B}^s_{p,q}(\mathcal{P}_\star)$. \\
 
\noindent {\it Step IV.} On the other hand, note that 
 $$f= \sum_{j} \sum_{j_0^\circ(Q^n_i)=j}  c_{Q^n_i} h_{Q^n_i}= \sum_{n} \sum_{i\leq r_n}  c_{Q^n_i} h_{Q^n_i} ,$$  
 where
 $$ c_{Q^n_i}= \frac{1}{2^ni^{1/q}}$$
 and
 $$h_{Q^n_i}= t^n_i \frac{1_{A^n_i}}{|A^n_i|^{1/p-s}} - z^n_i \frac{1_{B^n_i}}{|B^n_i|^{1/p-s}},$$
 with 
\begin{equation}\label{tn} t^n_i= \frac{1}{|P^n_i|^{1/p-s-1/2}}  \Big(\frac{1}{|A^n_i|}+ \frac{1}{|B^n_i|}\Big)^{-1/2} |A^n_i|^{1/p-s-1}\leq \Cll{pppu},\end{equation}
\begin{equation}\label{zn}z^n_i= \frac{1}{|P^n_i|^{1/p-s-1/2}}  \Big(\frac{1}{|A^n_i|}+ \frac{1}{|B^n_i|}\Big)^{-1/2} |B^n_i|^{1/p-s-1}\leq \Crr{pppu}.\end{equation}
Here $\Crr{pppu}$ depends only on the geometry of $\mathcal{P}_\star$,  $s$ and $p$. 
We have
\begin{eqnarray*} \Big( \sum_j \big(   \sum_{j_0^\circ(Q^n_i)=j} c_{Q^n_i}^p \big)^{q/p}  \Big)^{1/q}&=&\Big( \sum_n  \frac{1}{2^{nq}} \big(   \sum_{i\leq  r_n }   \frac{1}{i^{p/q}}\big)^{q/p}  \Big)^{1/q}\\
&=&\big( \sum_n  \frac{1}{2^{nq}} \big)^{1/q}  \big(   \sum_{i}   \frac{1}{i^{p/q}}\big)^{1/p} < \infty.
\end{eqnarray*} 
For every  $k\in \mathbb{N}$ define
$$\mathcal{Y}^k(Q^n_i)=\mathcal{P}^k_\circ \cap \bigcup_\ell \big( \mathcal{F}^\ell(A^n_i)\cup \mathcal{F}^{-\ell}(A^n_i)\cup \mathcal{F}^\ell(B^n_i)\cup \mathcal{F}^{-\ell}(B^n_i)\big),$$
where the families $\mathcal{F}^\ell$ are as defined in Proposition \ref{fam} taking $\mathcal{P}=\mathcal{P}_\circ$. Of course $\mathcal{Y}^k(Q^n_i)=\emptyset$ for $k< v_n$, and for every $P\in \mathcal{Y}^k(Q^n_i)$ we have $P\subset Q^n_i$, and consequently 
$$ |P|\leq   \Crr{maior}^{k-v_n}|Q^n_i|   \leq \Crr{th} \Crr{menor}^{-\Crr{cj}}\Crr{maior}^{k-v_n} \min\{ |A^n_i|, |B^n_i|  \}.$$ 
Moreover by Proposition \ref{fam}.D
$$\#\mathcal{Y}^k(Q^n_i)\leq \frac{8}{\Crr{menor}}.$$
Then
\begin{eqnarray*}
h_{Q^n_i}&=& t^n_i \frac{1_{A^n_i}}{|A^n_i|^{1/p-s}} - z^n_i \frac{1_{B^n_i}}{|B^n_i|^{1/p-s}}\\
&&=\sum_{k} \ \  \Big[ \sum_{\substack{P\in \mathcal{Y}^k(Q^n_i)\\ P\subset A^n_i} } t^n_i \Big( \frac{|P|}{|A^n_i|}\Big)^{1/p-s} a_P   \Big] - \Big[ \sum_{\substack{P\in \mathcal{Y}^k(Q^n_i) \\P\subset B^n_i}} z^n_i \Big( \frac{|P|}{|B^n_i|}\Big)^{1/p-s} a_P   \Big] \\
&&=\sum_{k} \ \   \sum_{P\in \mathcal{Y}^k(Q^n_i)} s_{P,Q^n_i} a_P   \\
\end{eqnarray*} 
Here $a_W$ is the canonical Souza's atom on $W$ and 
$$s_{P,Q^n_i}=  t^n_i \Big( \frac{|P\cap A^n_i|}{|A^n_i|}\Big)^{1/p-s}  - z^n_i \Big( \frac{|P\cap B^n_i|}{|B^n_i|}\Big)^{1/p-s}.$$
Consequently 
\begin{eqnarray*} &&\sum_{P\in \mathcal{Y}^k(Q^n_i)} |s_{P,Q^n_i}|^p \\
&\leq&2^{p} \Big[ \sum_{\substack{P\in \mathcal{Y}^k(Q^n_i)\\ P\subset A^n_i} }  (t^n_i )^p \Big( \frac{|P|}{|A^n_i|}\Big)^{1-sp}    \Big] +2^p \Big[ \sum_{\substack{P\in \mathcal{Y}^k(Q^n_i) \\P\subset B^n_i}}  (z^n_i)^p \Big( \frac{|P|}{|B^n_i|}\Big)^{1-sp}    \Big]\\
&\leq& \Cll{pppp} \Crr{maior}^{(k-v_n)(1-sp)}.
\end{eqnarray*} 
So by Proposition \ref{besov-trans}.A  in \cite{smania-besov}(take $\mathcal{G}=\mathcal{W}=\mathcal{P}_\circ$, $A=B=0$ and $k_i=i$) we have that $f \in \mathcal{B}^s_{p,q}(\mathcal{P}_\circ)$. This completes the proof of Claim 5. \\ \\

\noindent {\it Claim 6.} Suppose that $q> p$ and 
$$\sup_k \# \{j^\star_0(Q), Q\in \mathcal{P}_\circ^k     \}=+\infty.$$
Then  there is  $f\in \mathcal{B}^s_{p,q}(\mathcal{P}_\star)\setminus \mathcal{B}^s_{p,q}(\mathcal{P}_\circ)$.\\ \\ 

The  proof of Claim 6. has many similarities with  the proof of Claim 5., however there are  modifications that may not seems obvious to the reader. We describe them below. \\ \\ 
\noindent {\it Step I.} Consider families   $\{Q^n_1,\dots, Q^n_{r_n}\}\subset \mathcal{P}_\circ^{v_n}$  as in the Step I. in the proof of  Claim 5,  {\it except}  that we replace condition (\ref{ccco1}) by
\begin{equation}\label{cooo} \sum_{1 \leq i\leq r_n}\frac{1}{i} > 2^{np}.\end{equation} 
and we also demand  $v_{n+1}> v_n+3\Crr{cj}+2 $. For every $n$ and $1\leq i\leq r_n$ choose  $P^n_i\in \mathcal{P}_\star^{j_0^\star(Q^n_i)}$ such that $P^n_i\subset Q^n_i$, and also 
$\hat{Q}^n_i\subset P^n_i$ satisfying $\hat{Q}^n_i\in \mathcal{P}_\circ^{j_0^\circ(P^n_i)}$. We have
$v_n\leq j_0^\circ(P^n_i)\leq v_n+\Crr{cj}.$\\
 Let $A^n_i,B^n_i \in \mathcal{P}_\circ^{j_0^\circ(P^n_i)+1}$ be the only children of $\hat{Q}^n_i$. Define $S^n_i=(A^n_i,B^n_i)$ and let $\phi_{S^n_i}$ be  the corresponding element of the unconditional basis of $L^{t}$, for every $1 < t < \infty$, defined in Section \ref{besov-haar} in \cite{smania-besov} using the grid $\mathcal{P}_\circ$.

 \noindent {\it Step II.} Define
\begin{equation}\label{order2} f=\sum_{j} \sum_{j_0^\star(P^n_i)=j} \frac{1}{2^ni^{1/p}|\hat{Q}^n_i|^{1/p-s-1/2}}\phi_{S^n_i}\end{equation}
We claim that $f$ is well defined and it belongs to $\mathcal{B}^s_{p,q}(\mathcal{P}_\star)$. Indeed, we can write
 $$f= \sum_{j} \sum_{j_0^\star(P^n_i)=j}  c_{P^n_i} h_{P^n_i},$$  
 where
 $$ c_{P^n_i}= \frac{1}{2^ni^{1/p}}$$
 and
 $$h_{P^n_i}= t^n_i \frac{1_{A^n_i}}{|A^n_i|^{1/p-s}} - z^n_i \frac{1_{B^n_i}}{|B^n_i|^{1/p-s}},$$
 with 
$$t^n_i= \frac{1}{|\hat{Q}^n_i|^{1/p-s-1/2}}  \Big(\frac{1}{|A^n_i|}+ \frac{1}{|B^n_i|}\Big)^{-1/2} |A^n_i|^{1/p-s-1}\leq \Cll{pppuu},$$
$$z^n_i= \frac{1}{|\hat{Q}^n_i|^{1/p-s-1/2}}  \Big(\frac{1}{|A^n_i|}+ \frac{1}{|B^n_i|}\Big)^{-1/2} |B^n_i|^{1/p-s-1}\leq \Crr{pppuu}.$$
So 
\begin{eqnarray*} \Big( \sum_j \big(   \sum_{j_0^\star(P^n_i)=j} c_{\hat{Q}^n_i}^p \big)^{q/p}  \Big)^{1/q}&=&\Big( \sum_n  \sum_{i\leq r_n} \frac{1}{2^{nq} i^{q/p}}   \Big)^{1/q}\\
&=&\big( \sum_n  \frac{1}{2^{nq}} \big)^{1/q}  \big(   \sum_{i}   \frac{1}{i^{q/p}}\big)^{1/q} < \infty.
\end{eqnarray*} 
For every  $k\in \mathbb{N}$ define
$$\mathcal{Y}^k(P^n_i)=\mathcal{P}^k_\star \cap \bigcup_\ell \mathcal{F}^\ell(A^n_i)\cup \mathcal{F}^{-\ell}(A^n_i)\cup \mathcal{F}^\ell(B^n_i)\cup \mathcal{F}^{-\ell}(B^n_i),$$
where the families $\mathcal{F}^\ell$ are as defined in Proposition \ref{fam} taking $\mathcal{P}=\mathcal{P}_\star$. Of course $\mathcal{Y}^k(P^n_i)=\emptyset$ for $k< j^\star_0(\hat{Q}^n_i)\leq  j^\star_0(P^n_i)$, and for every $P\in \mathcal{Y}^k(P^n_i)$ we have $P\subset \hat{Q}^n_i\subset P^n_i$, and consequently due Proposition \ref{compi}.B 
$$ |P|\leq   \Crr{maior}^{k-j^\star_0(P^n_i)}|P^n_i| \leq \Crr{th} \Crr{maior}^{k-j^\star_0(P^n_i)}  |\hat{Q}^n_i|   \leq \Crr{th} \Crr{menor}^{-1}\Crr{maior}^{k-j^\star_0(P^n_i)} \min\{ |A^n_i|, |B^n_i|  \}.$$ 
Moreover
$$\#\mathcal{Y}^k(P^n_i)\leq \frac{8}{\Crr{menor}}.$$
Then
\begin{eqnarray*}
h_{P^n_i}&=& t^n_i \frac{1_{A^n_i}}{|A^n_i|^{1/p-s}} - z^n_i \frac{1_{B^n_i}}{|B^n_i|^{1/p-s}}\\
&&=\sum_{k\in \mathbb{N}} \ \  \Big[ \sum_{\substack{P\in \mathcal{Y}^k(P^n_i)\\ P\subset A^n_i} } t^n_i \Big( \frac{|P|}{|A^n_i|}\Big)^{1/p-s} a_P   \Big] - \Big[ \sum_{\substack{P\in \mathcal{Y}^k(P^n_i) \\P\subset B^n_i}} z^n_i \Big( \frac{|P|}{|B^n_i|}\Big)^{1/p-s} a_P   \Big] \\
&&=\sum_{k\in \mathbb{N}} \ \   \sum_{P\in \mathcal{Y}^k(P^n_i)} s_{P,P^n_i} a_P   \\
\end{eqnarray*} 
Here $a_W$ is the canonical Souza's atom on $W$ and 
$$s_{P,P^n_i}=  t^n_i \Big( \frac{|P\cap A^n_i|}{|A^n_i|}\Big)^{1/p-s}  - z^n_i \Big( \frac{|P\cap B^n_i|}{|B^n_i|}\Big)^{1/p-s}.$$
Consequently 
\begin{eqnarray*} &&\sum_{P\in \mathcal{Y}^k(P^n_i)} |s_{P,P^n_i}|^p \\
&\leq&2^{p} \Big[ \sum_{\substack{P\in \mathcal{Y}^k(P^n_i)\\ P\subset A^n_i} }  (t^n_i )^p \Big( \frac{|P|}{|A^n_i|}\Big)^{1-sp}    \Big] +2^p \Big[ \sum_{\substack{P\in \mathcal{Y}^k(P^n_i) \\P\subset B^n_i}}  (z^n_i)^p \Big( \frac{|P|}{|B^n_i|}\Big)^{1-sp}    \Big]\\
&\leq& \Cll{pppw} \Crr{maior}^{(k-j^\star_0(P^n_i))(1-sp)}.
\end{eqnarray*} 
So by Proposition \ref{besov-trans}.A  in \cite{smania-besov} (take $\mathcal{G}=\mathcal{W}=\mathcal{P}_\star$, $A=B=0$ and $k_i=i$) we have that $f \in \mathcal{B}^s_{p,q}(\mathcal{P}_\star)$. In particular $f\in L^t$, for some  $t > 1$. \\

\noindent {\it Step III.} In particular $f\in L^t$, for some  $t > 1$. Since $\{\phi_{S^n_i}\}$ belongs to a unconditional basis of $L^t$ we can write 
$$f=\sum_{j} \sum_{j_0^\circ(\hat{Q}^n_i)=j} \frac{1}{2^ni^{1/p}|\hat{Q}^n_i|^{1/p-s-1/2}}\phi_{S^n_i}$$
Since $v_{n+1}> v_n+3\Crr{cj}+2 $,  
$v_n\leq j_0^\circ(P^n_i)\leq v_n+\Crr{cj}$ and  $\hat{Q}^n_i\neq \hat{Q}^n_\ell$ for $i\neq \ell$  we conclude that 
$$\# \{ \hat{Q}^n_i \colon v_n\leq j^\circ_0(\hat{Q}^n_i)\leq v_n+\Crr{cj}\}=r_n,$$
so there is at least one $j\in [v_n,v_n+\Crr{cj}]$ such that 
$$ \sum_{\hat{Q}^n_i\in \mathcal{P}^{j}_\circ}     \frac{1}{i}\geq  \frac{2^{np}}{\Crr{cj}+1}.$$
We conclude that 
\begin{eqnarray*}
 N_{haar}^\circ (f)&=&\Big( \sum_{j\geq 1}  \big(  \sum_{\hat{Q}^n_i\in \mathcal{P}^{j}_\circ}    \frac{1}{2^{np}i } \big)^{q/p} \Big)^{1/q},\\
 &=&\Big( \sum_{n} \frac{1}{2^{nq}}  \sum_{v_n\leq j\leq v_n+\Crr{cj}} \Big( \sum_{\hat{Q}^n_i\in \mathcal{P}^{j}_\circ}     \frac{1}{i}\Big)^{q/p}  \Big)^{1/q}\\
 &=&\Big( \sum_{n} \frac{1}{2^{nq}}  \sum_{v_n\leq j\leq v_n+\Crr{cj}} \Big( \sum_{\hat{Q}^n_i\in \mathcal{P}^{j}_\circ}     \frac{1}{i}\Big)^{q/p}  \Big)^{1/q}=\infty,
 \end{eqnarray*}
 so $f \notin \mathcal{B}^s_{p,q}(\mathcal{P}_\circ)$. \\

This completes the proof of Claim 6. \\ \\

To complete the proof of Proposition \ref{inc}, note that Claim 4. tell us that $C\Rightarrow A$ and $C\Rightarrow B$. Suppose that $A$ holds, that is $\mathcal{B}^s_{p,q}(\mathcal{P}_\circ) \subset  \mathcal{B}^s_{p,q}(\mathcal{P}_\star)$.  If $q< p$  then Claim 5. implies that (\ref{bou}) holds. So  by Claim 3. we have that $ h_\circ^{-1}\circ h_\star$ is a bi-Lipchitz function and Claim 4. gives $\mathcal{B}^s_{p,q}(\mathcal{P}_\circ)=\mathcal{B}^s_{p,q}(\mathcal{P}_\star)$. On the other hand if $p< q$  then  Claim 6  (exchanging the roles of $\mathcal{P}_\circ$ and $\mathcal{P}_\star$ )  implies that
$$\sup_k \# \{j^\circ_0(Q), Q\in \mathcal{P}_\star^k     \}< +\infty,$$
so again  by Claim 3. we have that $ h_\circ^{-1}\circ h_\star$ is a bi-Lipchitz function and Claim 4. gives $\mathcal{B}^s_{p,q}(\mathcal{P}_\circ)=\mathcal{B}^s_{p,q}(\mathcal{P}_\star)$. We conclude that $A\Rightarrow B$ and $A\Rightarrow C$. Obviously $B\Rightarrow A$, so this concludes the proof. 
 \end{proof}

\begin{proposition}[Exotic Besov-ish spaces]  Let $\mathbb{B}^s_{p,q}$ be the Besov space of $[0,1]$ as defined in Section \ref{beshom}.  There are quasisymmetric grids $\mathcal{P}_\star$ such that $\mathcal{B}^s_{p,q}(\mathcal{P}_\star)\neq \mathbb{B}^s_{p,q}$. 
\end{proposition} 
\begin{proof} Let $\mathbb{S}^1=\mathbb{R}/\mathbb{Z}$ and consider the smooth expanding map $f\colon \mathbb{S}^1\rightarrow \mathbb{S}^1$ given by $f(x)=2x \ mod \ 1.$ Note that the dyadic grid on $[0,1]$  is a sequence of Markov partitions for $f$, that is, $f: P \rightarrow f(P)$ is a difeomorphism for every $P\in \mathcal{D}^{k}$, with $k\geq 1$, and $f(P)\in \mathcal{D}^{k-1}$.  Let $g\colon \mathbb{S}^1 \rightarrow \mathbb{S}^1$ be  a $2$ to $1$ $C^2$-covering of the circle such that $g(0)=0$. Then there is a quasisymmetric homeomorphism $h_\star \colon \mathcal{S}^1\rightarrow \mathbb{S}^1$ such that $h_\star(0)=0$ and conjugates $f$ and $g$, that is, and $h_\star\circ g=  f \circ h_\star$ on $\mathbb{S}^1$. In particular  there is a quasisymmetric grid $\mathcal{P}_\star$ such that $h_\star(\mathcal{P}_\star)=\mathcal{D}$. We can choose $g$ in such way that for some $n\geq 1$ there is some $n$-periodic point $p$ of $g$  ($g^n(p)=p$) satisfying 
$$ Dg^n(p)\neq Df^n(h_\star(p))=2^n.$$
We claim that $h_\star$ is {\t not} a bi-Lipchitz map. Indeed, otherwise the conjugacy $h_\star$ would be absolutely continuous with respect to the Lebesgue measure. But Shub and Sullivan \cite{shusu} proved that every absolutely continuous conjugacy between two $C^2$ expanding maps  on the circle is indeed a $C^1$ difeomorphism. But this is not possible since in this case it is easy to show that $Dg^n(p)= Df^n(h_\star(p))$ for {\it every} $n$ periodic point $p$ of $g$.  Since $h_\star$ is not a bi-Lipchitz map, Proposition \ref{inc} (take $h_\circ = Id$ and $\mathcal{P}_\circ=\mathcal{D}$)  implies that $\mathcal{B}^s_{p,q}(\mathcal{P}_\star)\neq \mathcal{B}^s_{p,q}(\mathcal{D})=\mathbb{B}^s_{p,q}$. 
\end{proof}

\begin{remark} \label{exotic2} One can ask  if $\mathcal{B}^s_{p,q}(\mathcal{P}_\star)$ above  is the Besov space of a homogenous space of the form $([0,1],\rho,m)$, where $m$ is the Lebesgue measure and $\rho$ is {\it some} metric on $[0,1]$ other than the euclidean one. We do not have a complete answer for that. However note that $\rho$ can not be the metric defined in (\ref{narch}). Indeed, if we consider a recalibration $\mathcal{G}$ of $\mathcal{P}_\star$ then we have that $\mathcal{G}$ is a good grid of intervals where every interval in the same level as more of less the same length. Using the same argument of the proof of Theorem \ref{mart} we conclude that $\mathcal{B}^s_{p,q}(\mathcal{G})$ is the Besov space $\mathbb{B}^s_{p,q}([0,1],d,m)=\mathcal{B}^s_{p,q}(\mathcal{D})$, where $d$ is the euclidean distance. 
\end{remark} 
\section{Good grid invariance for  $\mathcal{B}^s_{p,p}$} 

The next result tell us  that when $p=q$ the Besov spaces $\mathcal{B}^s_{p,p}$ are far more resilient to modifications  of the grid, what somehow remind us of Theorem 2.2 in Vodop$\prime$yanov \cite{podo}. See also Koch, Koskela, Saksman and Soto \cite{koskela}.

\begin{proposition} Consider the interval $I=[0,1]$ with the Lebesgue measure $m$. Let   $\mathcal{P}_\circ, \ \mathcal{P}_\star$ be  grids of $(I,m)$ such that 
\begin{itemize}
\item[A.] We have that $\mathcal{P}_\circ^k, \mathcal{P}_\star^k$are  families of intervals for every  $k$, 
\item[B.] For every $k\neq j$ we have $\mathcal{P}_\circ^k\cap  \mathcal{P}_\circ^j= \emptyset$ and
\item[C.]  We have  that  $\mathcal{P}_\star$  is a $(\Crr{menor},\Crr{maior})$-good grid. 
\end{itemize}
Then $$\mathcal{B}^s_{p,p}(\mathcal{P}_\circ)\subset \mathcal{B}^s_{p,p}(\mathcal{P}_\star)$$ and the inclusion is continuous.  As a consequence  $\mathcal{B}^s_{p,p}(\mathcal{P}_\star)= B^s_{p,p}$, where $B^s_{p,p}$ is the Besov space in the homogeneous  space $(I,m)$.
\end{proposition} 
\begin{proof}We will prove it in three steps. \\ \\
\noindent{\it Step I.} Note that given two grids  $\mathcal{G}_1$ and $\mathcal{G}_2$ such that 
$$\mathcal{G}_i^k\cap \mathcal{G}_i^j=\emptyset$$
for every $k\neq j$, $i=1,2$, and
$$\bigcup_k \mathcal{G}_1^k=\bigcup_k \mathcal{G}_2^k$$
then 
$\mathcal{B}^s_{p,p}(\mathcal{G}_1)=\mathcal{B}^s_{p,p}(\mathcal{G}_2)$
and the corresponding norms are equivalent. 

Define $\mathcal{G}_i=(\mathcal{G}_i^k)_k$, with $i=\star,\circ$, in the following way.
$$\mathcal{G}_i^k=\{ P\in  \bigcup_j \mathcal{P}_i^j \ s.t. \     \frac{1}{2^{k+1}}  <  |P|\leq \frac{1}{2^k}  \}.$$
Then $\mathcal{G}_i$ is a grid and consequently $\mathcal{B}^s_{p,p}(\mathcal{G}_i)=\mathcal{B}^s_{p,p}(\mathcal{P}_i)$ and the corresponding norms are equivalent.  It remains to show the (continuous) inclusion $\mathcal{B}^s_{p,p}(\mathcal{G}_\circ)\subset \mathcal{B}^s_{p,p}(\mathcal{G}_\star)$.\\ \\

\noindent {\it Step II.}  Since $\mathcal{P}_\star$ is a good grid there is $r \geq1$ such that for every $j\in \mathbb{N}$ and $x\in I$ there is $P\in \mathcal{G}_\star^i$, with $i\in [j+1,j+r]$, satisfying  $x\in P$. \\

\noindent {\it Step III.}  Given $Q\in \mathcal{G}^{k_0}_\circ$, define 
$$j_0=\min \{ j \in \mathbb{N},\   \text{ there is } P \in \mathcal{G}_\star^j \text{ satisfying } P\subset Q  \}.$$
By Step II  we have that $ k_0+r+2\leq   j_0 \leq k_0$. Let
$$\tilde{\mathcal{F}}^{j_0}(Q)=\{ P\in  \mathcal{G}^{j_0}_\star, \ int \ P\subset Q \}$$
and
$$A_{j_0}=\bigcup_{P\in \tilde{\mathcal{F}}^{j_0}(Q)} \overline{P}$$
and by induction define
$$\tilde{\mathcal{F}}^{j}(Q)=\{ P\in  \mathcal{G}^j_\star, \ int \ P\subset Q\setminus \cup_{i<j} A_i  \}.$$
and
$$A_j= \bigcup_{P\in \tilde{\mathcal{F}}^{j}(Q)} \overline{P}$$
Since $\mathcal{P}_\star$ is a nested sequence of partitions we can find a subfamily $\mathcal{F}^{j}(Q)\subset \tilde{\mathcal{F}}^{j}(Q)$ such that 
$$A_j=\bigcup_{P\in \mathcal{F}^{j}(Q)} \overline{P}$$
and $\mathcal{F}^{j}(Q)$ is a family of pairwise disjoint intervals. By definition $\{A_j\}_i$ is a family of sets with  pairwise  disjoint interior. Note that $$D_j=\cup_{i\leq j} A_i.$$ can be described as
$$D_j = \{ x\in Q \text{ s.t. there is $P\in  \mathcal{G}^i_\star$, with $i\leq j$, such that $x\in P\subset Q$}\}.$$
Let $[a_j,b_j]$ be the convex hull of $D_j$.
If $x \in [a_j,b_j]\setminus D_j$ then 
$$dist(x, \{a_j, b_j\})> \frac{1}{2^j},$$
so by Step II there is $P\in \mathcal{G}_\star^i$, with $i \in [j+1,j+r]$, such that $x\in P\subset [a_j,b_j]\subset Q.$ Note that $P\cap W=\emptyset$ for every $W\in \tilde{\mathcal{F}}^{j}(Q)$, otherwise we would have  $x \in P\subset W\subset A_j$. In particular $x\in \mathcal{F}^i(Q)$. Consequently
$$[a_j,b_j]\subset  D_{j+r}=\cup_{i\leq j+r} A_i.$$
Additionally if $y\in \{a_j, b_j\}$ satisfies
$$dist(y,\partial Q) > \frac{1}{2^{j-r}}$$
then by Step III there is $P\in \mathcal{G}_\star^i$, with $i \in [j-r+1,j]$, such that $P \subset Q$  and $P\cap ([0,a_j[\cup ]b_j,1])\neq \emptyset,$ which contradicts the inclusion $D_j\subset  [a_j,b_j]$. We conclude that
$$m(Q\setminus D_{j+r}) < \frac{1}{2^{j-r-1}},$$
so
$$\#\mathcal{F}^{j+r+1}(Q)\leq 2^{j+r+2} m(Q\setminus D_{j+r}) \leq 2^{2r+1}.$$
 and for every $\ell$ satisfying $j_0\leq \ell \leq j_0+r$ we have 
 $$\#\mathcal{F}^{\ell}(Q)\leq 2^{\ell+1} m(Q)\leq 2^{r+1},$$
 so $\#\mathcal{F}^{\ell}(Q)\leq  2^{2r+1}$ for every $\ell\geq j_0$. 
If  $a_W$ is the canonical  Souza's atom on $W$ and $Q\in \mathcal{G}^{k_0}_\circ$ we have
$$a_Q=\sum_{j\geq j_0} \sum_{P\in \mathcal{F}^j(Q)}\Big(\frac{|P|}{|Q|}\Big)^{1/p-s} a_P,$$
and for every $j\geq j_0$
$$\sum_{P\in \mathcal{F}^j(Q)} \Big(\frac{|P|}{|Q|}\Big)^{1-sp} \leq   2^{2r+1} \Big( \frac{2^{-j}}{2^{-(k_0+1)}} \Big)^{1-sp}\leq 2^{2r+1-(r+2)(1-sp)} 2^{-(j-j_0)(1-sp)}$$
By Propositon \ref{besov-trans}.C in \cite{smania-besov} we have that $\mathcal{B}^s_{p,p}(\mathcal{G}_\circ)\subset \mathcal{B}^s_{p,p}(\mathcal{G}_\star)$ and moreover  this inclusion is continuous.  In particular $\mathcal{B}^s_{p,p}(\mathcal{G}_\star)=\mathcal{B}^s_{p,p}(\mathcal{D})$, and  by Propostition \ref{besovreal} we have $\mathcal{B}^s_{p,p}(\mathcal{D})=B^s_{p,p}$.
\end{proof}

\bibliographystyle{abbrv}
\bibliography{bibliografiab}

\def\cprime{$'$}
\begin{thebibliography}{10}

\bibitem{qsm2}
H.~Aimar, B.~Iaffei, and L.~Nitti.
\newblock On the {M}ac\'\i as-{S}egovia metrization of quasi-metric spaces.
\newblock {\em Rev. Un. Mat. Argentina}, 41(2):67--75, 1998.

\bibitem{sharp}
R.~Alvarado and M.~Mitrea.
\newblock {\em Hardy spaces on {A}hlfors-regular quasi metric spaces}, volume
  2142 of {\em Lecture Notes in Mathematics}.
\newblock Springer, Cham, 2015.
\newblock A sharp theory.

\bibitem{smania-transfer}
A.~Arbieto and D.~Smania.
\newblock Transfer operators and atomic decomposition.
\newblock arXiv 1903.06943, 2019.

\bibitem{besov}
O.~V. Besov.
\newblock On some families of functional spaces. {I}mbedding and extension
  theorems.
\newblock {\em Dokl. Akad. Nauk SSSR}, 126:1163--1165, 1959.

\bibitem{change2}
G.~Bourdaud.
\newblock Changes of variable in {B}esov spaces. {II}.
\newblock {\em Forum Math.}, 12(5):545--563, 2000.

\bibitem{change1}
G.~Bourdaud and W.~Sickel.
\newblock Changes of variable in {B}esov spaces.
\newblock {\em Math. Nachr.}, 198:19--39, 1999.

\bibitem{christ}
M.~Christ.
\newblock A {$T(b)$} theorem with remarks on analytic capacity and the {C}auchy
  integral.
\newblock {\em Colloq. Math.}, 60/61(2):601--628, 1990.

\bibitem{cw}
R.~R. Coifman and G.~Weiss.
\newblock {\em Analyse harmonique non-commutative sur certains espaces
  homog\`enes}.
\newblock Lecture Notes in Mathematics, Vol. 242. Springer-Verlag, Berlin-New
  York, 1971.
\newblock \'Etude de certaines int\'egrales singuli\`eres.

\bibitem{souzao1}
G.~S. de~Souza.
\newblock The atomic decomposition of {B}esov-{B}ergman-{L}ipschitz spaces.
\newblock {\em Proc. Amer. Math. Soc.}, 94(4):682--686, 1985.

\bibitem{souzao2}
G.~S. de~Souza.
\newblock Two more characterizations of {B}esov-{B}ergman-{L}ipschitz spaces.
\newblock {\em Real Anal. Exchange}, 11(1):75--79, 1985/86.
\newblock The ninth summer real analysis symposium (Louisville, Ky., 1985).

\bibitem{sn}
G.~S. de~Souza, R.~O'Neil, and G.~Sampson.
\newblock Several characterizations for the special atom spaces with
  applications.
\newblock {\em Rev. Mat. Iberoamericana}, 2(3):333--355, 1986.

\bibitem{spline2}
R.~A. DeVore and V.~A. Popov.
\newblock Interpolation of {B}esov spaces.
\newblock {\em Trans. Amer. Math. Soc.}, 305(1):397--414, 1988.

\bibitem{fj}
M.~Frazier and B.~Jawerth.
\newblock Decomposition of {B}esov spaces.
\newblock {\em Indiana Univ. Math. J.}, 34(4):777--799, 1985.

\bibitem{martingale}
C.~Gu and M.~Taibleson.
\newblock Besov spaces on nonhomogeneous martingales.
\newblock In {\em Harmonic analysis and discrete potential theory ({F}rascati,
  1991)}, pages 69--84. Plenum, New York, 1992.

\bibitem{han2}
Y.~Han, S.~Lu, and D.~Yang.
\newblock Inhomogeneous {B}esov and {T}riebel-{L}izorkin spaces on spaces of
  homogeneous type.
\newblock {\em Approx. Theory Appl. (N.S.)}, 15(3):37--65, 1999.

\bibitem{hs}
Y.~S. Han and E.~T. Sawyer.
\newblock Littlewood-{P}aley theory on spaces of homogeneous type and the
  classical function spaces.
\newblock {\em Mem. Amer. Math. Soc.}, 110(530):vi+126, 1994.

\bibitem{hk}
T.~Hyt\"onen and A.~Kairema.
\newblock Systems of dyadic cubes in a doubling metric space.
\newblock {\em Colloq. Math.}, 126(1):1--33, 2012.

\bibitem{koskela}
H.~Koch, P.~Koskela, E.~Saksman, and T.~Soto.
\newblock Bounded compositions on scaling invariant {B}esov spaces.
\newblock {\em J. Funct. Anal.}, 266(5):2765--2788, 2014.

\bibitem{koskela2}
P.~Koskela, D.~Yang, and Y.~Zhou.
\newblock Pointwise characterizations of {B}esov and {T}riebel-{L}izorkin
  spaces and quasiconformal mappings.
\newblock {\em Adv. Math.}, 226(4):3579--3621, 2011.

\bibitem{qsm1}
R.~A. Mac\'ias and C.~Segovia.
\newblock Lipschitz functions on spaces of homogeneous type.
\newblock {\em Adv. in Math.}, 33(3):257--270, 1979.

\bibitem{oswald0}
P.~Oswald.
\newblock On function spaces related to finite element approximation theory.
\newblock {\em Z. Anal. Anwendungen}, 9(1):43--64, 1990.

\bibitem{oswald}
P.~Oswald.
\newblock {\em Multilevel finite element approximation}.
\newblock Teubner Skripten zur Numerik. [Teubner Scripts on Numerical
  Mathematics]. B. G. Teubner, Stuttgart, 1994.
\newblock Theory and applications.

\bibitem{qsm3}
M.~Paluszy\'nski and K.~Stempak.
\newblock On quasi-metric and metric spaces.
\newblock {\em Proc. Amer. Math. Soc.}, 137(12):4307--4312, 2009.

\bibitem{peetre}
J.~Peetre.
\newblock {\em New thoughts on {B}esov spaces}.
\newblock Mathematics Department, Duke University, Durham, N.C., 1976.
\newblock Duke University Mathematics Series, No. 1.

\bibitem{pu}
F.~Przytycki and M.~Urba\'nski.
\newblock {\em Conformal fractals: ergodic theory methods}, volume 371 of {\em
  London Mathematical Society Lecture Note Series}.
\newblock Cambridge University Press, Cambridge, 2010.

\bibitem{shusu}
M.~Shub and D.~Sullivan.
\newblock Expanding endomorphisms of the circle revisited.
\newblock {\em Ergodic Theory Dynam. Systems}, 5(2):285--289, 1985.

\bibitem{smania-besov}
D.~Smania.
\newblock Besov-ish spaces through atomic decomposition.
\newblock arXiv 1903.06937, 2019.

\bibitem{stein}
E.~M. Stein.
\newblock {\em Singular integrals and differentiability properties of
  functions}.
\newblock Princeton Mathematical Series, No. 30. Princeton University Press,
  Princeton, N.J., 1970.

\bibitem{ns}
H.~Triebel.
\newblock Non-smooth atoms and pointwise multipliers in function spaces.
\newblock {\em Ann. Mat. Pura Appl. (4)}, 182(4):457--486, 2003.

\bibitem{fractal2}
H.~Triebel.
\newblock A new approach to function spaces on quasi-metric spaces.
\newblock {\em Rev. Mat. Complut.}, 18(1):7--48, 2005.

\bibitem{tbook}
H.~Triebel.
\newblock {\em Theory of function spaces}.
\newblock Modern Birkh\"auser Classics. Birkh\"auser/Springer Basel AG, Basel,
  2010.
\newblock Reprint of 1983 edition [MR0730762], Also published in 1983 by
  Birkh\"auser Verlag [MR0781540].

\bibitem{fractal}
H.~Triebel.
\newblock {\em Fractals and spectra}.
\newblock Modern Birkh\"auser Classics. Birkh\"auser Verlag, Basel, 2011.
\newblock Related to Fourier analysis and function spaces.

\bibitem{u}
M.~Urba\'nski.
\newblock Measures and dimensions in conformal dynamics.
\newblock {\em Bull. Amer. Math. Soc. (N.S.)}, 40(3):281--321, 2003.

\bibitem{podo}
S.~K. Vodop$\prime$yanov.
\newblock Mappings of homogeneous groups and embeddings of function spaces.
\newblock {\em Sibirsk. Mat. Zh.}, 30(5):25--41, 215, 1989.

\bibitem{yang}
D.~Yang.
\newblock Besov spaces on spaces of homogeneous type and fractals.
\newblock {\em Studia Math.}, 156(1):15--30, 2003.

\end{thebibliography}

\end{document}